\documentclass[12pt,reqno]{amsart}
\usepackage{amssymb,amsthm,amsmath,amstext,amsxtra}
\usepackage{fullpage}
\usepackage[all]{xy}
\usepackage{booktabs}
\usepackage{mathtools}
\usepackage{hyperref}
\hypersetup{colorlinks=true,urlcolor=blue,citecolor=blue,linkcolor=blue}
\usepackage{enumerate}
\usepackage{ stmaryrd }
\usepackage{enumitem}
\usepackage{comment}
\usepackage{multirow}
\usepackage{colonequals}
\usepackage{tikz}
\usepackage{marginnote}
\usepackage[export]{adjustbox}
\usepackage{cellspace}
\usepackage{tikz-cd}

\newcommand{\bigboxplus}{
  \mathop{
    \vphantom{\bigoplus} 
    \mathchoice
      {\vcenter{\hbox{\resizebox{\widthof{$\displaystyle\bigoplus$}}{!}{$\boxplus$}}}}
      {\vcenter{\hbox{\resizebox{\widthof{$\bigoplus$}}{!}{$\boxplus$}}}}
      {\vcenter{\hbox{\resizebox{\widthof{$\scriptstyle\oplus$}}{!}{$\boxplus$}}}}
      {\vcenter{\hbox{\resizebox{\widthof{$\scriptscriptstyle\oplus$}}{!}{$\boxplus$}}}}
  }\displaylimits 
}

\newcommand{\legen}[2]{\left(\frac{#1}{#2}\right)}

\numberwithin{equation}{subsection}

\DeclarePairedDelimiter\abs{\lvert\hspace{0.1ex}}{\rvert}
 
\theoremstyle{plain}
\newtheorem*{theorem*}{Theorem}
\newtheorem{theorem}[equation]{Theorem}

\newtheorem{proposition}[equation]{Proposition}
\newtheorem{lemma}[equation]{Lemma}
\newtheorem{corollary}[equation]{Corollary}
\newtheorem{conjecture}[equation]{Conjecture}
\newtheorem{thm}[equation]{Theorem}
\newtheorem{conj}[equation]{Conjecture}

\theoremstyle{definition}
\newtheorem{definition}[equation]{Definition}

\theoremstyle{remark}
\newtheorem{remark}[equation]{Remark}
\newtheorem{example}[equation]{Example}

\newenvironment{enumalph}
{\begin{enumerate}}
{\end{enumerate}}
\newenvironment{enumroman}
{\begin{enumerate}}
{\end{enumerate}}

\newcommand{\Z}{\mathbb{Z}}
\newcommand{\bZ}{\Z}
\newcommand{\Q}{\mathbb{Q}}
\newcommand{\bQ}{\Q}
\newcommand{\R}{\mathbb{R}}

\newcommand{\F}{\mathbb{F}}

\newcommand{\A}{\mathbb{A}}

\newcommand{\cN}{\mathcal{N}}
\newcommand{\calO}{\mathcal{O}}
\newcommand{\cO}{\calO}

\newcommand{\spdual}{{}^{*}}

\newcommand{\defi}[1]{\textsf{#1}} 	

\DeclareMathOperator{\coker}{coker}
\DeclareMathOperator{\Cl}{Cl}

\DeclareMathOperator{\Hom}{Hom}

\DeclareMathOperator{\Prob}{Prob}
\DeclareMathOperator{\sgnrk}{sgnrk}

\DeclareMathOperator{\Nm}{Nm}
\DeclareMathOperator{\Isom}{Isom}

\DeclareMathOperator{\Gal}{Gal}
\DeclareMathOperator{\Aut}{Aut}

\DeclareMathOperator{\ord}{ord}

\DeclareMathOperator{\Tr}{\mathrm{Tr}}

\DeclareMathOperator{\rk}{\mathrm{rk}}

\DeclareMathOperator{\Sel}{\mathrm{Sel}}
\DeclareMathOperator{\sgn}{\mathrm{sgn}}

\newcommand{\fraka}{\mathfrak{a}}

\newcommand{\frakp}{\mathfrak{p}}

\newcommand{\Magma}{{\textsc{Magma}}}

\makeatletter
\newcommand{\bigperp}{%
  \mathop{\mathpalette\bigp@rp\relax}%
  \displaylimits
}

\def\VR{\kern-\arraycolsep\strut\vrule &\kern-\arraycolsep}
\def\vr{\kern-\arraycolsep & \kern-\arraycolsep}

\newcommand{\bigp@rp}[2]{%
  \vcenter{
    \m@th\hbox{\scalebox{\ifx#1\displaystyle2.1\else1.5\fi}{$#1\perp$}}
  }%
}
\makeatother

\newcommand{\frakaI}{\mathfrak{a}}  
\newcommand{\ClfourK}{\Cl_4(K)}     
\newcommand{\ClfourKplus}{\Cl_4^+(K)}

\newcommand{\rkchi}{\rk_{\chi}}     
\newcommand{\chidual}{\chi\spdual}  
\DeclareMathOperator{\im}{img}      
\newcommand{\Ftwobar}{\overline{\F}_2} 
\newcommand{\Kbar}{\overline{K}}    
\newcommand{\vrho}{r} 		
\newcommand{\Cond}{{\rm Cond}}

\usepackage{tikz}
\usepackage{tikz-cd}
\usepackage{xcolor}
\definecolor{darkred}{HTML}{CC1F1F}
\definecolor{green}{rgb}{.4,.7,.4}
\definecolor{blue}{rgb}{.2,.6,.75}
\definecolor{pastelyellow}{rgb}{0.992157, 0.552941, 0.235294}
\definecolor{pastelorange}{rgb}{0.941176, 0.231373, 0.12549}
\definecolor{pastelred}{rgb}{0.741176, 0., 0.14902}
\definecolor{darkbrown}{rgb}{0.25098, 0., 0.0745098}
\usepackage{hyperref}
\hypersetup{pdftitle={On unit signatures and narrow class groups of odd degree abelian number fields:  structure and heuristics},
pdfauthor={Benjamin Breen, Ila Varma, and John Voight}}
\hypersetup{colorlinks=true,linkcolor=blue,anchorcolor=blue,citecolor=blue}

\begin{document}

\title[Heuristics for odd abelian fields]{On unit signatures and narrow class groups of odd degree abelian number fields}

\author{Benjamin Breen}
\email{benjamin.k.breen.gr@dartmouth.edu}
\address{Department of Mathematics, Dartmouth College, 6188 Kemeny Hall, Hanover, NH 03755}
\urladdr{} 

\author{Ila Varma}
\email{ila@math.toronto.edu}
\address{Department of Mathematics, University of Toronto, Bahen Centre, 40 St.~George Street, Toronto, ON M5S 2E4}
\urladdr{} 

\author{John Voight \\ (appendix with Noam Elkies)}
\email{jvoight@gmail.com}
\address{Department of Mathematics, Dartmouth College, 6188 Kemeny Hall, Hanover, NH 03755}
\urladdr{\url{http://www.math.dartmouth.edu/~jvoight/}} 

\subjclass{11R29, 11R27, 11R45, 11Y40}

\begin{abstract}
For an abelian number field of odd degree, we study the structure of its $2$-Selmer group as a bilinear space and as a Galois module.  We prove structural results and make predictions for the distribution of unit signature ranks and narrow class groups in families where the degree and Galois group are fixed.  
\end{abstract}

\date{\today} 
\maketitle
\setcounter{tocdepth}{1} 
\tableofcontents 

\section{Introduction}

\subsection{Motivation}  \label{sec:motivation}

Originating in the study of solutions to the negative Pell equation, the investigation of signatures of units in number rings dates back at least to Lagrange. While a considerable amount of progress has been made for quadratic fields \cite{Stevenhagen,FouvryKluners,CKMP}, predictions for the distribution of narrow class groups and possible signs of units under real embeddings for certain families of higher degree number fields have only recently been developed \cite{DV, DDK, Breen}.

\smallskip

In this paper, we study unit signatures and class groups of abelian number fields of odd degree.  To illustrate and motivate our results, we begin with two special cases. 
\begin{conj}[{Conjecture \ref{conj:ell35sgkr0}}]\label{conj:cubicfields}
As $K$ varies over cyclic cubic number fields, the probability that $K$ has a totally positive system of fundamental units is approximately $3\%$.
\end{conj}

For the conjectures presented in this paper,  we sidestep the issue of ordering fields:
we expect that any \emph{fair} counting function in the terminology of Wood \cite{Woodprob} should be allowed, for example ordering by conductor or by the norm of the product of ramified primes. 
We are led to Conjecture \ref{conj:cubicfields} by combining structural results established herein with a randomness hypothesis (\hyperref[item:H2]{H2}) in the vein of the Cohen--Lenstra heuristics.  
This conjecture agrees well with computational evidence (see section \ref{sec:cubfields}); the following theorem provides additional theoretical support.

\begin{thm}[{Theorem \ref{thm:cyclfields}, with Elkies}] \label{thm:thmb}
There exist infinitely many cyclic cubic fields with a totally positive system of fundamental units. 
\end{thm} 

The proof of Theorem \ref{thm:thmb} involves the study the integral points on a log K3 surface.  The (infinite) family of \emph{simplest cubic fields} of Shanks were each shown to have units of all possible signatures by Washington \cite[p.~371]{Washington}, the case complementary to Theorem \ref{thm:thmb}. 

Our second illustrative conjecture is as follows.

\begin{conj}[{Conjecture \ref{mainconj::Tworanks0}}] \label{conj:just7clplus}
As $K$ ranges over cyclic number fields of degree $7$ with odd class number, the probability that the narrow class number is also odd is $7/9$.
\end{conj}

We recall that the narrow class group of a field coincides with the class group if and only if there are units of all possible signatures.  This conjecture also matches computational evidence well (see section \ref{sec:septiccomp}).  

\medskip

The predictions above are based on the philosophy underlying the Cohen--Lenstra heuristics, which predicts random behavior for arithmetic objects \emph{as soon as} one accounts for all of the determined structure.  Early examples of the need to account for structural properties, including genus theory and ranks of units, were already present in the original paper of Cohen--Lenstra \cite{CohenLenstra}.  It remains mysterious and important to understand how one must account for additional structure in generalizations of the Cohen--Lenstra heuristics.  For example, what makes a prime \emph{good} \cite{CohenMartinet-good}, and the interaction between $p$-parts of class groups and the presence of $p$th roots of unity \cite{Malle,Malle1,AdamMalle}, remain unresolved.  On the other hand, some reflection principles like those of Scholz and Leopoldt \cite{Leopoldt:Spiegelungsatz}, seem to be inherently compatible with the Cohen--Lenstra--Martinet conjectures \cite{Dutarte,Lee1,Lee2}.

In this paper, we propose a model for the distribution of $2$-parts of narrow class groups and signatures of units in families of abelian number fields of fixed odd degree and Galois group.  For such families, the Galois action and the presence of the $2$nd roots of unity suggest additional, nontrivial structure to account for (as confirmed by computations and available function field analogues).  The requirement that the degree is odd when $p = 2$ isolates the ``roots of unity problem'' from other obstructions to arithmetic randomness, including genus theory.  Since the narrow class group is an extension of the class group by an elementary abelian $2$-group that measures signatures of units, our efforts are concentrated on $2$-parts.

Our contributions are thus twofold.  First, for these families we precisely identify and analyze relevant structure, including the relationship to reflection principles.  Second, under the hypothesis that what remains behaves randomly, we make exact predictions for the behavior of units and class groups---with corroborating computational evidence.

\subsection{Structure: class groups}

We now set up the structures we study and model in this paper.  We build on work of Dummit--Voight \cite{DV}, who make predictions for fields of odd degree $n$ whose Galois closure has Galois group $S_n$.  Here, we instead consider Galois extensions of odd degree.

Recall that the \defi{$2$-Selmer group} of a number field $K$ is
\begin{equation*}\label{eqdef:2Sel}
\Sel_2(K) \colonequals \{ \alpha \in K^{\times} : (\alpha) = \frakaI^2 \text{ for a fractional ideal } \frakaI \mbox{ of }K\} /(K^{\times})^2.
\end{equation*} 
Attached to $K$ is a finite-dimensional $\F_2$-vector space $V_\infty(K) \boxplus V_2(K)$ equipped with a nondegenerate symmetric bilinear form and a homomorphism
	$$\varphi_K \colon \Sel_2(K) \rightarrow V_\infty(K) \boxplus V_2(K)$$
called the \defi{$2$-Selmer signature map}. Dummit--Voight \cite{DV} showed that the image $S(K) \colonequals \im(\varphi_K)$ of $\varphi_K$ is a maximal, totally isotropic subspace.  If $K$ is Galois over $\Q$ with Galois group $G_K$, then we observe that the above objects carry a $G_K$-action and in particular $S(K)$ is a $G_K$-invariant, maximal totally isotropic subspace (Corollary \ref{cor:Gkinv}). 

In preparation for stating our guiding result, we introduce a bit of notation.  Let $G$ be a finite abelian group of odd order and let $\chi$ be an $\Ftwobar$-character of $G$.  
Then every irreducible $\F_2[G]$-module is isomorphic to $\F_2(\chi)$, the value field of an \defi{$\Ftwobar$-character} $\chi \colon G \rightarrow \Ftwobar^\times$ taking values in a (fixed) algebraic closure $\Ftwobar$ of $\F_2$, where $G$ acts through the character $\chi$. For a  finitely generated $\Z[G]$-module $M$, write $\rkchi M$ for the multiplicity of the irreducible module $\F_2(\chi)$ in the $\F_2[G]$-module $M/M^2$, and let $\rk_2 M \colonequals \dim_{\F_2} M/M^2$.  For an $\Ftwobar$-character $\chi$ of $G$, there is a noncanonical $\F_2[G]$-module isomorphism $\Hom_{\F_2}(\F_2(\chi),\F_2) \simeq \F_2(\chi^{-1})$ (see Lemma \ref{lem:homisom} and the discussion preceding it), and we write $\chidual \colonequals \chi^{-1}$ for the corresponding dual character. We say $\chi$ is \defi{self-dual} if $\F_2(\chidual) \simeq \F_2(\chi)$ as $\F_2[G]$-modules.
For an $\Ftwobar$-character $\chi$ of $G$ and $V$ an $\F_2[G]$-module, we write $V_\chi$ for the $\F_2(\chi)$-isotypic component of $V$ and $V_{\chi^{\pm}} \colonequals V_\chi + V_{\chidual}$ for the sum.
If $V$ is equipped with a symmetric, $G$-invariant $\F_2$-bilinear form, then the decomposition of $V$ into the spaces $V_{\chi^{\pm}}$ is orthogonal (Lemma \ref{lem:OrthoDecomp1}), giving a canonical decomposition as $\F_2[G]$-modules.


Now let $K$ be a Galois number field with abelian Galois group $G_K$ of odd order; then the class group $\Cl(K)$ and narrow class group $\Cl^+(K)$ are $\Z[G_K]$-modules.  For an $\Ftwobar$-character $\chi$, we define the following nonnegative integers:
\begin{equation} \label{eqn:notationrhow}
  \begin{aligned} 
  \rho_{\chi}(K) &\colonequals \rk_\chi \Cl(K); \\
  \rho_{\chi}^+(K) &\colonequals \rk_\chi \Cl^+(K); \\
  k_{\chi}^+(K) &\colonequals \rk_\chi \Cl^+(K) - \rk_\chi\Cl(K) \\
  &\ = \rho^+_\chi(K)-\rho_\chi(K). 
    \end{aligned}
\end{equation}
We refer to $k_\chi^+(K)$ as the $\chi$-\defi{isotropy rank}.  When $K$ is clear from context, we drop it from the notation.  Our main theorem, governing the above structures and quantities, is as follows. 

\begin{thm}[{Theorem \ref{mainthm::Representationtheory}}] \label{thm:KGal}
Let $K$ be a Galois number field with abelian Galois group $G_K$ of odd order.  Then for each $\Ftwobar$-character $\chi$, there are exactly $6$ possibilities for $S(K)_{\chi^{\pm}} \leq V_\infty(K) \boxplus V_2(K)$ up to $G_K$-equivariant isometry.
\end{thm}

Theorem \ref{thm:KGal} follows from an investigation of the $\F_2[G_K]$-module structure of the $2$-Selmer signature map together with a classification of invariant, maximal isotropic subspaces in a bilinear space with group action.  The six possibilities are given in Table \ref{table:cases0}: we write $q \colonequals \#\F_2(\chi)=\#\F_2(\chidual)$, and we write $\#\Isom_G(V)$ for the number of $G$-equivariant isometries of $V(K) \colonequals V_\infty(K) \boxplus V_2(K)$.  All cases occur (see Example \ref{exm::cases7}), so the statement is optimal in this sense. 

In Tables \ref{table:cases0} and  \ref{table:cases1}  we observe parallel relations when $V_\infty(K)$ is replaced by $V_2(K)$ and $\Cl^+(K)$ is replaced by $\ClfourK$, the ray class group of $K$ of conductor $4$, with the quantities $\rho_{4,\chi}(K)$ and $k_{4,\chi}(K)$ defined analogously as in \eqref{eqn:notationrhow}; we restrict attention to narrow class groups in this introduction.
 
{\setlength{\tabcolsep}{4pt}\renewcommand\arraystretch{1.1}
\begin{equation} 
\label{table:cases0}\addtocounter{equation}{1} \notag
\begin{gathered}
\begin{tabular}[m]{c|c||c|c|c|c}
Case & $\chi$ self-dual? & \multicolumn{1}{c|}{$S_{\chi^\pm}$} & \multicolumn{1}{c|}{$S_{\chi^\pm} \cap V_\infty$ }  & \multicolumn{1}{c|}{$S_{\chi^\pm} \cap V_2$ } & $\#\Isom_G(V)$ \\[2pt]
 \hline \hline 
\textsf{A} & Yes & $\F_2(\chi)$ & $\{0\}$ & $\{0\}$ & $\sqrt{q}+1$  \\ \hline
\textsf{B} & No & $\F_2(\chi)^2$ & $\F_2(\chi)$ & $\F_2(\chi)$ & $1$   \\
\textsf{B}${}^\prime$ & No & $\F_2(\chidual)^2$ & $\F_2(\chidual)$  & $\F_2(\chidual)$ & $1$   \\
 \textsf{C} & No & $\F_2(\chi) \oplus \F_2(\chidual)$ & $\F_2(\chi)$ & $\F_2(\chidual)$ & $1$ \\ 
 \textsf{C}${}^\prime$ & No & $\F_2(\chi) \oplus \F_2(\chidual)$ & $\F_2(\chidual)$ & $\F_2(\chi)$ & $1$  \\ 
 \textsf{D}  &No & $\F_2(\chi) \oplus \F_2(\chidual)$ & $\{0\}$ & $\{0\}$ & $q-1$  \\ 
\end{tabular} \\
\text{Table \ref{table:cases0}: Possibilities for the Galois bilinear structure of the image of the $2$-Selmer group}
\end{gathered}
\end{equation}}

{\setlength{\tabcolsep}{4pt}\renewcommand\arraystretch{1.1}
\begin{equation} \label{table:cases1}\addtocounter{equation}{1} \notag
\begin{gathered}
\begin{tabular}[m]{c||c|c|c|c|c|c|c}
Case & \multicolumn{1}{Sc|}{$\rho_\chi$ and $\rho_{\chidual} $} & \multicolumn{1}{Sc|}{$\rho_\chi^+$ and $\rho_{\chidual}^+ $} & \multicolumn{1}{Sc|}{$\rho_{4,\chi}$ and $\rho_{4,\chidual} $}  & $k^+_\chi $ & $k^+_{\chidual}$ & $k_{4,\chi} $ & $k_{4,\chidual}$ \\[2pt]
 \hline \hline 
\textsf{A}  & $\rho_{\chi} = \rho_{\chidual}$ & $\rho_{\chi}^+ = \rho_{\chidual}^+$ & $\rho_{4,\chi} = \rho_{4,\chidual} $ & 0 & 0 & 0 & 0  \\ \hline
\textsf{B} & $\rho_\chi  = \rho_{\chidual}  +1$  & $\rho_\chi^+   = \rho_{\chidual}^+$ & $\rho_{4,\chi} = \rho_{4,\chidual} $ & 0 & 1 & 0 & 1  \\
\textsf{B}${}^\prime$ & $\rho_\chi   = \rho_{\chidual}  -1$ & $\rho_\chi^+   = \rho_{\chidual}^+ $& $\rho_{4,\chi} = \rho_{4,\chidual} $ & 1&  0 & 1 & 0  \\
 \textsf{C} & $\rho_\chi   = \rho_{\chidual} $  & $\rho_\chi^+   = \rho_{\chidual}^+  -1$ & $\rho_{4,\chi} = \rho_{4,\chidual} +1 $  & 0 & 1 &  1 & 0   \\ 
 \textsf{C}${}^\prime$ & $\rho_\chi   = \rho_{\chidual}  $  &$\rho_\chi^+   = \rho_{\chidual}^+ +1$ & $\rho_{4,\chi} = \rho_{4,\chidual} -1 $  & 1& 0 & 0 & 1  \\ 
 \textsf{D} & $\rho_\chi   = \rho_{\chidual} $  & $\rho_\chi^+   = \rho_{\chidual}^+$   &  $\rho_{4,\chi} = \rho_{4,\chidual} $  & 0  & 0 & 0&0  \\ 
\end{tabular} \\
\text{Table \ref{table:cases1}: Possibilities for the class group and isotropy rank}
\end{gathered}
\end{equation}}

The following corollary is then immediate.

\begin{corollary} \label{cor:thmKgalcor_reflect}
Under the hypotheses of Theorem \textup{\ref{thm:KGal}}, we have
\[ \abs{\rk_\chi \Cl(K) - \rk_{\chidual} \Cl(K)} \leq 1, \]
\[ \abs{\rk_\chi \Cl^+(K) - \rk_{\chidual} \Cl^+(K)} \leq 1, \]
and $0 \leq k_{\chi}^+ + k_{\chidual}^+ \leq 1$.  
\end{corollary}

Corollary \ref{cor:thmKgalcor_reflect} can be seen as a Spiegelungssatz or \emph{reflection theorem} as in Leopoldt \cite{Leopoldt:Spiegelungsatz} for $p=2$, and therefore Theorem \ref{thm:KGal} can be seen as a precise refinement of it.  A precursor to Corollary \ref{cor:thmKgalcor_reflect} is the theorem of Armitage--Fr\"ohlich \cite{ArmitageFrohlich}, generalized by Taylor \cite{Taylor} and Oriat \cite{Oriat1,Oriat2}.  Gras then proved a very general $T$-$S$-reflection principle \cite[Th\'eor\`eme 5.18]{Gras} (see also the presentation in his book \cite[Chapter II, Theorem 5.4.5]{Grasbook}); however, certain corollaries for $p=2$ \cite[Chapter II, Corollary 5.4.6(ii)]{Grasbook} (details \cite[Chapter II, (5.4.9)]{Grasbook} added in the second printing) are incorrect: case $\textsf{D}$ of Table \ref{table:cases0} does not appear. 

We show that rank inequalities like Corollary \ref{cor:thmKgalcor_reflect} for a Galois number field $K$ of odd degree follow from Kummer duality and the $G_K$-module structure of the $2$-Selmer group (and its intersection with coordinate subspaces in the $2$-Selmer signature space).  In particular, the relevant reflection principles are \emph{already encoded}.  In particular, we recover easily several classical results from the literature.  Our results can also instead be seen to fit into a much more general context (Poitou--Tate duality of Selmer groups); however, in view of the subtleties indicated in the previous paragraph, one advantage of our approach is it provides a self-contained, uniform, and transparent proof of these corollaries.  At the same time, the concrete description in Theorem \ref{thm:KGal} states the precise structure (in particular, the image of the $2$-Selmer group under the signature map is a $G_K$-invariant, maximal totally isotropic subspace) which must be respected in a random model and thereby serves as the foundation for our heuristics, which we present in sections \ref{sec:heuristicsintro}--\ref{sec:heuristicsintro_units}. 

\subsection{Structure: units}

The structural result in Theorem \ref{thm:KGal} has the following consequence for units.  Let $\calO_K$ be the ring of integers of $K$.  The \defi{archimedean signature} map 
$\sgn_\infty\colon K^\times \rightarrow \prod_{v \mid \infty} \{\pm 1\} \simeq \F_2^{n}$ 
is the surjective group homomorphism recording the signs of elements of $K^\times$ under each real embedding; its kernel $K_{>0}^\times~\colonequals~\ker(\sgn_\infty)$ is the group of \defi{totally positive} elements of $K$.  Let $\calO_{K,>0}^\times \colonequals \calO_K^\times \cap K_{>0}^\times$ denote the group of \defi{totally positive units}.  Define 
\begin{equation} \label{eq:sgnrkchi}
\sgnrk_{\chi}(\calO_K^\times)  \colonequals \rk_\chi \sgn_\infty(\calO_K^\times), \\
\end{equation}
and the \defi{unit signature rank} of $K$
\begin{equation}
\sgnrk(\calO_K^\times) \colonequals \dim_{\F_2} \sgn_\infty(\calO_K^\times) = \textstyle{\sum_\chi} \sgnrk_{\chi}(\calO_K^\times)\cdot [\F_2(\chi):\F_2],
\end{equation}
where the sum indexes over isomorphism classes of $\Ftwobar$-characters $\chi$.  
The structure on unit signature ranks imposed by the Galois module structure is summarized in the following result, keeping the notation \eqref{eqn:notationrhow}.

\begin{thm}[Theorem \ref{thm:oddm10}] \label{thm:oddm1}
Let $K$ be an abelian number field of odd degree with Galois group $G_K$, and let $\chi$ be an $\Ftwobar$-character of $G_K$.  Then the following statements hold. 
\begin{enumalph}
\item If $k_{\chi}^+(K) = 1$, then $\sgnrk_{\chi}(\calO_K^\times) = 0$.
\item If $k_{\chi}^+(K) = 0$, then $1- \rk_{\chi}\Cl(K) \leq \sgnrk_{\chi}(\calO_K^\times) \leq 1$.
\end {enumalph}
\end{thm}
When the degree of $K$ is prime, summing over $\chi$ gives the following corollary.  
\begin{corollary}[Corollary \ref{cor:Kcycsgrkyup}]
Let $K$ be a cyclic number field of odd prime degree $\ell$, and let $f$ be the order of $2$ modulo $\ell$.  Then
\[ \sgnrk(\calO_K^\times) \equiv 1 \pmod{f}, \]
and the following statements hold.
\begin{enumalph}
\item If $f$ is odd, then ${\textstyle \frac{\ell +1}{2}} - \rk_2\Cl(K)  \leq \sgnrk(\calO_K^\times) \leq \ell;$
\item If $f$ is even, then $\ell - \rk_2\Cl(K) \leq \sgnrk(\calO_K^\times) \leq \ell.$
\end{enumalph}
\end{corollary}

For example, if $2$ is a primitive root modulo $\ell$ and the class number of $K$ is odd, then $\sgnrk(\calO_K^\times) = \ell$; this result for $\ell =3$ was observed by Armitage--Fr\"ohlich \cite[Theorem V]{ArmitageFrohlich}. 

\subsection{Heuristics: narrow class groups} \label{sec:heuristicsintro}

We begin by applying the results in the previous section to make predictions for narrow class groups and signatures of units for odd-degree abelian number fields.  We keep the notation of \eqref{eqn:notationrhow}.

\smallskip

Let $G$ be a finite abelian group of odd order.  A \defi{$G$-number field} is a Galois number field $K$, inside a fixed algebraic closure of $\Q$, equipped with an isomorphism $G_K \simeq G$, where $G_K \colonequals \Gal(K\,|\,\Q)$.  Such a field $K$ is totally real, so $\pm1$ are the only roots of unity in $K$.  

Returning to Theorem \ref{thm:KGal} and Table \ref{table:cases0}, we see that the quantities $k_\chi^+,k_{\chidual}^+$ are uniquely determined by $S_{\chi^\pm}$ in the cases where $\chi$ is self-dual or cases \textsf{B} and \textsf{B}${}^\prime$ when $\chi$ is not self-dual.  However, when $\chi$ is not self-dual and $\rho_\chi=\rho_{\chi^*}$, there is a question about the distribution of cases \textsf{C}, \textsf{C}${}^\prime$, and \textsf{D}.  Modeling the image of $2$-Selmer signature map as a random totally isotropic $G$-invariant subspace in the $2$-Selmer signature space (see heuristic assumption (\hyperref[item:H1]{H1})), we are led to the following conjecture.

\begin{conj}[Conjecture \ref{mainconj::Tworanks0}]\label{mainconj::Tworanks} 
 Let $G$ be an abelian group of odd order, and let $\chi$ be an $\Ftwobar$-character of $G$ that is \emph{not}~self-dual and let $q \colonequals \#\F_2(\chi)$.  Then as $K$ varies over $G$-number fields {satisfying $\rho_\chi(K)=\rho_{\chidual}(K)$}, 
\begin{equation} \label{eqn:probkwkw} 
\begin{array}{rcl}
\Prob \bigl(k^+_\chi(K)+k^+_{\chidual}(K) = 0 \bigr)   &=& \displaystyle\frac{q-1 }{ q +1}; \vspace{5pt}\\ 
\Prob \bigl (k^+_\chi(K)+k^+_{\chidual}(K) = 1 \bigr ) &=& \displaystyle \frac{2 }{q +1}. 
\end{array}
\end{equation}
\end{conj}

A concrete application of Conjecture \ref{mainconj::Tworanks} is given in Conjecture \ref{conj:3mod4}, as follows.  Suppose $2$ has order $(\ell-1)/2$ modulo a prime $\ell \equiv 7 \pmod{8}$: then there are exactly two non-self-dual characters, and if $\Cl(K)$ is self-dual then $k_\chi=k_{\chidual}=0$.  So as $K$ varies over cyclic number fields of degree $\ell$ such that $\Cl(K)[2]$ is self-dual, Conjecture \ref{mainconj::Tworanks} predicts that
\begin{equation}\label{eq:3mod4}
\Prob\bigl(\Cl^+(K)[2]\simeq\Cl(K)[2]) = \frac{2^{\frac{\ell-1}{2}} - 1}{2^{\frac{\ell-1}{2} }+ 1}.\vspace{5pt}
\end{equation}
We further expect that the probability in Conjecture \ref{mainconj::Tworanks} remains the same in certain natural subfamilies, such as when we fix the value $\rk_\chi\Cl(K)=\rk_{\chidual}\Cl(K)=\vrho$.  As a special case, we arrive at Conjecture \ref{conj:just7clplus}.

\subsection{Heuristics: units} \label{sec:heuristicsintro_units}

Next, we make predictions for signatures of units.  Our model can be applied under many scenarios; in this introduction, we consider two simple, illustrative cases.  We first examine the situation when the degree is prime and the class number is odd.  Modeling $\calO_K^\times/(\calO_K^{\times})^2$ as a random $G_K$-invariant subspace of the 2-Selmer group of $K$ containing $-1$, and under an independence hypothesis (\hyperref[item:H2prime]{H2${}^\prime$}), 
we are led to the following conjecture.

\begin{conj}[Conjecture \ref{conj:ellm1fq}]
Let $\ell$ be an odd prime such that the order $f$ of $\,2$ in $(\bZ/\ell\bZ)^\times$ is odd. Let $q\colonequals 2^f$, and define $m \colonequals {\frac{\ell-1}{2f}} \in \Z_{>0}$.  
Then as $K$ varies over cyclic number fields of degree $\ell$ with \emph{odd} class number, 
\[ \Prob\bigl(\sgnrk(\calO_K^\times)=fs+{\textstyle \frac{\ell-1}{2}}\bigr) =  \binom{m}{s} \left ( \frac{q -1 }{q +1} \right )^s \left (   \frac{2 }{q +1} \right )^{m-s} \]
for $0 \leq s \leq m$.
\end{conj}

Second, we consider the situation when $\ell = 3$ or $5$ with no additional assumption on the class number. In this case, Corollary \ref{cor:Kcycsgrkyup}(b) implies that $\sgnrk(\calO_K^\times)=1$ or $\ell$.  
Although complete heuristics for the 2-part of the class group over abelian fields are not known, Malle \cite{Malle} provides results in the case that $\ell = 3$ or $5$. We use the following notation: for $m \in \Z_{\geq 0} \cup \{\infty\}$ and $q \in \R_{>1}$, write $(q)_0 \colonequals 1$ and otherwise $(q)_m \colonequals \prod_{i=1}^m (1-q^{-i})$. Combining these results with a uniform random hypothesis
(\hyperref[item:H2]{H2}), we make the following prediction.  

 \begin{conj}[Conjecture \ref{conj:ell35sgkr0}] \label{conj:ell35sgkr}
Let $\ell = 3$ or $5$ and $q = 2^{\ell-1}$. As $K$ varies over cyclic number fields of degree $\ell$, then
\begin{equation} \label{eqn:alignedsgnrkE}
    \begin{aligned}
\Prob\bigl(  \sgnrk(\calO_K^\times) = 1  \bigr)  &= \left(1+\frac{1}{\sqrt{q}}\right)\frac{(\sqrt{q})_\infty (q^2)_\infty}{(q)_\infty^2} \cdot   \sum_{r = 0}^\infty  \frac{(\ell-1)^{\,r}}{q^{(r^2 + 3r)/2} \cdot (q)_r}   \cdot \frac{q^{r}-1}{q^{r+1}-1}.
\end{aligned}
\end{equation}
\end{conj}

Computing the numerical value of the quantity in (\ref{eqn:alignedsgnrkE}), we predict that approximately $3\%$ of cyclic cubic fields have $\sgnrk(\calO_K^\times)=1$ which yields Conjecture \ref{conj:cubicfields}.
For cyclic quintic fields we predict that this proportion drops to below $0.1\%$.  The predictions in the two conjectures above agree with the computational evidence we compiled: see section \ref{sec:calcs}.  

\begin{remark} \label{rmk:immaterial}
To extend the above conjectures to all odd primes (or more generally, to all abelian groups $G$ of odd order), we would need to refine the heuristics of Malle \cite{Malle,Malle1} to predict 
the distribution of $\rk_{\chi} \Cl(K)$.  This distribution will  depend on the representation theory of $\bZ/\ell\bZ$ (or more generally, of $G$); in particular, the constraints in Theorem \ref{thm:KGal} must be respected.  In contrast, when $2$ is a primitive root modulo $\ell$, there is only one nontrivial (necessarily self-dual) $\F_2[\bZ/\ell\bZ]$-module, so these representation-theoretic complexities are immaterial; in this case, we expect that the generalization of the above conjectures to such $\ell$ to be more straightforward.
\end{remark}

We expect that as $\ell \to \infty$ varies over odd primes, we have $\Prob\bigl(\sgnrk(\calO_K^\times) = 1\bigr) \to 0\%$, and we plan to give evidence to support this limiting behavior in the future (see also Remark \ref{rem:malle}).  

\begin{remark}
The statements we prove and conjecture above on unit signature ranks in odd degree extensions are quite different than the situation for real quadratic fields, related to solutions to the negative Pell equation.  By genus theory, 100\% of real quadratic fields have a totally positive unit \cite{FouvryKluners}, and the conjectural asymptotic due to Stevenhagen \cite{Stevenhagen} arises from an apparently different heuristic involving R\'edei matrices. 
\end{remark}

\begin{remark}
We are not aware of a function field analogue which would bear on the conjectures presented in this section.  These conjectures are based on structural properties of the $2$-Selmer signature map, which rely in an essential way on the fact that $2 \in \calO_K$ is neither a unit nor zero.  
\end{remark}

\subsection{Outline} 

In section \ref{sec:notation}, we set up basic notation and background.  In section \ref{sec:galmodstruct}, we study these objects in general as Galois modules over $\F_2$.  We then restrict to the case of odd Galois extensions in section \ref{sec:galmodstruct_odd} and show how reflection principles follow from the Galois action and Kummer duality---these are for completeness (and to indicate that they are not missing from our model).  We then further restrict to abelian extensions and in section \ref{sec:Galois} prove our main structural result, and we see classical reflection principles as a corollary.  In section \ref{sec:models} we introduce our heuristic assumptions and present our conjectures, including details on the low-degree cases. In section \ref{sec:calcs}, we carry out computations that provide some experimental evidence for our conjectures. Finally, in appendix \ref{sec:infincub} we prove Theorem \ref{thm:thmb}.

\subsection{Acknowledgements} 

The authors would like to thank Edgar Costa, David Dummit, Noam Elkies, Georges Gras, Brendan Hassett, Hershy Kisilevsky, Evan O'Dorney, Arul Shankar, Jared Weinstein, and Melanie Matchett Wood for comments, and Tommy Hofmann for sharing his list of cyclic septic fields.  Special thanks go to two anonymous referees for their excellent and detailed feedback and suggestions.
Varma was partially supported by an NSF MSPRF Grant (DMS-1502834) and an NSF Grant (DMS-1844206).   Voight was supported by an NSF CAREER Award (DMS-1151047) and a Simons Collaboration Grant (550029).  Elkies was partially supported by an NSF grant (DMS-1502161) and a Simons Collaboration Grant.

\section{Properties of the 2-Selmer group and its signature spaces} \label{sec:notation}

We begin by setting up some notation and recalling basic definitions and previous results.  

\subsection{Basic notation}

If $A$ is a (multiplicatively written) abelian group and $m \in \Z_{>0}$, we write 
\[ A[m] \colonequals \{a \in A : a^m=1\} \]
for the $m$-torsion subgroup of $A$. For a prime $p$, we write
 $$\rk_p (A) \colonequals \dim_{\F_p} (A/A^p)$$ 
 for the \defi{$p$-rank} of $A$; we then have $\#A[p] = p^{\rk_p(A)}$.
 
\smallskip

We quickly prove a standard lemma (for lack of a reference).  Let $\Z_{(p)} \colonequals \{a/b \in \Q : p \nmid b\}$ be the localization of $\Z$ away from a prime $p$.

\begin{lemma}\label{lem:genfact}
Let $G$ be a finite group, let $p \nmid \#G$ be prime, and let $M$ be a
finitely generated, torsion $\Z_{(p)}[G]$-module.  Then there is a
(noncanonical) isomorphism $M/pM \xrightarrow{\sim} M[p]$ as $\F_p[G]$-modules.
\end{lemma}

\begin{proof}
Recall (by Maschke's theorem) that every finitely generated
$\F_p[G]$-module is semisimple, since $p \nmid \#G$.  
Let $m=p^r$ be the exponent of $M$ (as an abelian group), with $r \geq
0$.  We argue by induction on $r$.   If $r \leq 1$, then $pM=\{0\}$ so
indeed $M/pM=M=M[p]$.

Suppose the result holds whenever $M$ has exponent dividing $p^r$; we
prove it for $M$ of exponent $p^{r+1}$.  Multiplication by $p$ gives
an exact sequence
\[ 0 \to M[p] \to M \to pM \to 0 \]
of $\Z_{(p)}[G]$-modules.  We can repeat this with $pM$, giving the
following diagram, with exact rows and columns:
\begin{equation}
\begin{tikzcd}
 & 0 \ar{d} & 0 \ar{d} & 0 \ar{d} \\
  0 \ar{r} & (pM)[p] \ar{d} \ar{r} & pM \ar{r} \ar{d}  & p^2M \ar{r}
\ar{d} & 0 \\
  0 \ar{r} & M[p] \ar{d} \ar{r} & M \ar{r} \ar{d} & pM \ar{r} \ar{d} & 0 \\
  0 \ar{r} & M_1 \ar{d} \ar{r} & M/pM \ar{r} \ar{d} & pM/p^2M \ar{r}
\ar{d} & 0 \\
 & 0 & 0 & 0
  \end{tikzcd}
  \end{equation}
Here, $M_1 \colonequals \coker((pM)[p] \to M[p]) \simeq M[p]/(pM)[p]$.
By semisimplicity, the left vertical and bottom horizontal maps split, so
\begin{equation}
\begin{aligned}
M[p] &\simeq (pM)[p] \oplus M_1 \\
M/pM &\simeq (pM/p^2M) \oplus M_1
\end{aligned}
\end{equation}
Since $pM$ has exponent $p^r$, by induction $(pM)[p] \simeq (pM/p^2M)$ so $M[p] \simeq M/pM$.
\end{proof}

\smallskip

Let $K$ be a number field of degree $n=[K:\Q]$ with $r_1$ real and $r_2$ complex places, with algebraic closure $\Kbar$ and with ring of integers $\cO_K$.  For a prime $p$, we denote the localization of $\cO_K$ away from $(p)$ by 
    $$\cO_{K,(p)} \colonequals \{\alpha \in K^\times : \ord_\frakp(\alpha) \geq 0 \text{ for all primes $\frakp \mid (p)$}\}
$$
and the completion of $\cO_K$ at $p$ by $\cO_{K,p} \colonequals \cO_K \otimes \bZ_{p}$, so that $\calO_{K,(p)} \hookrightarrow \calO_{K,p}$.  For a place $v$ of $K$, we let $K_v$ denote the completion of $K$ at $v$ and $\calO_{K,v}$ its valuation ring, and we let $$(\_,\_)_v \colon K_v^\times \times K_v^\times \to \{\pm 1\}$$ denote the Hilbert symbol at $v$: recall that for $\alpha_v$ and $\beta_v \in K_v^\times$, we have $(\alpha_v,\beta_v)_v = 1$ if and only if $\beta_v$ is in the image of the norm map from $K_v[x]/(x^2-\alpha_v)$ to $K_v$.  

\subsection{The 2-Selmer group and its signature spaces} \label{sec:2selsigdefs}
 
The main object of study is the \defi{$2$-Selmer group} of a number field $K$, defined as
\begin{equation*}
\Sel_2(K) \colonequals \{ \alpha \in K^{\times} : (\alpha) = \frakaI^2 \text{ for a fractional ideal } \frakaI \mbox{ of }K\} /(K^{\times})^2.
\end{equation*}
 Following Dummit-Voight \cite[Section 3]{DV}, we recall two signature spaces that keep track of behavior at $\infty$ and at $2$, as follows.

\begin{definition}\label{def:VooK} The \defi{archimedean signature space} $V_\infty(K)$ is defined as
\begin{equation*} 
V_\infty(K) \colonequals \prod_{\text{$v$ real}} \{\pm 1\} \simeq \prod_{v \mid \infty} K_v^\times/K_v^{\times 2}
\end{equation*}
where the second product runs over all real places of $K$.  The \defi{archimedean signature map} is 
\begin{equation*}
    \begin{aligned}
\sgn_\infty \colon K^\times &\to V_\infty(K) \\
\alpha &\mapsto (\sgn v(\alpha))_v
\end{aligned}
\end{equation*}
where $\sgn x = x/\abs{x}$ for $x \in \R^\times$.
\end{definition}
By definition, $\ker \sgn_\infty = K_{>0}^\times$, the \defi{totally positive elements} of $K$, which contains $(K^{\times})^2$, and so the map $\sgn_\infty$ induces a well-defined map $\varphi_{K,\infty}: \Sel_2(K) \rightarrow V_\infty(K)$.  The product of Hilbert symbols defines a map
\begin{equation*} 
\begin{aligned} 
b_\infty \colon V_\infty(K) \times V_\infty(K) & \to \{\pm 1\} \\  
\end{aligned} 
\end{equation*}
which is a (well-defined) symmetric, non-degenerate $\F_2$-bilinear form. 
 
\begin{definition}\label{def:V2} The \defi{2-adic signature space} $V_2(K)$  is defined as
\begin{equation*} 
V_2(K) \colonequals \calO_{K,(2)}^\times/(1+4\calO_{K,(2)})(\calO_{K,(2)}^{\times})^2.
\end{equation*}
The \defi{$2$-adic signature map} is the map
\begin{equation*}
 \sgn_2 \colon \calO_{K,(2)}^\times \to V_2(K)
 \end{equation*}
obtained from 
the projection $\cO_{K,(2)}^\times \rightarrow V_2(K)$.
\end{definition}

For the following statements we refer to Dummit--Voight \cite[\S 4]{DV}.  
We have $\dim_{\F_2}V_2(K)=n$ 
and there is an isomorphism of abelian groups
	\begin{equation*}
	V_2(K) \simeq \prod_{v \mid (2)} \cO_{K,v}^\times/ (1 + 4 \cO_{K,v})(\cO_{K,v}^\times)^2.
	\end{equation*}
		 Under this identification, the product of Hilbert symbols defines a map
\begin{equation*}
\begin{aligned} 
b_2 \colon V_2(K) \times V_2(K) & \to \{\pm 1\}    \\
 \big( (\alpha_v)_v, (\beta_v)_v \big) &\mapsto \prod_{v \mid (2)}  (\alpha_v,\beta_v)_v,   
\end{aligned} 
\end{equation*}
which is a nondegenerate, symmetric $\F_2$-bilinear form on $V_2(K)$. 
Every class in $\Sel_2(K)$ has a representative $\alpha$ such that $\alpha \in \calO_{K,(2)}^\times$, unique up to multiplication by an element of $(\cO_{K,(2)}^{\times})^2$;  
therefore, the map $\sgn_2$ induces a well-defined map $\varphi_{K,2}\colon \Sel_2(K) \rightarrow V_2(K)$.
  
\medskip

Putting these together, we define the \defi{2-Selmer signature space} as the orthogonal direct sum
\begin{equation*}
V(K) \colonequals V_\infty(K) \boxplus V_2(K)
\end{equation*}
and write $b \colonequals b_\infty \perp b_2$ for the bilinear form on $V(K)$.  The {isometry group} of $(V(K),b)$ is the product of the isometry groups (or equivalently, the subgroup of the total isometry group preserving each factor).
Equipped with $b$, the $2$-Selmer signature space $V(K)$ is a nondegenerate symmetric bilinear space over $\F_2$ of dimension $r_1+n$.  Similarly, we define the \defi{$2$-Selmer signature map} 
\begin{equation} \label{eqn:varphiKdef}
\varphi_K \colonequals \varphi_{K,\infty} \perp \varphi_{K,2} \colon \Sel_2(K) \to V(K). 
\end{equation}

\begin{theorem}[{Dummit--Voight \cite[Theorem 6.1]{DV}}] \label{Thm:MaximalIsotropic}
For a number field $K$, the image of the $2$-Selmer signature map $\varphi_K$ is a maximal totally isotropic subspace.
\end{theorem}

Recall from the introduction that the class group of $K$ is denoted by $\Cl(K)$, its narrow class group is denoted by $\Cl^+(K)$, and its ray class group of conductor 4 by $\Cl_4(K)$.

\begin{definition} \label{defn:isotropyranks}
The \defi{archimedean isotropy rank} of a number field $K$ is 
$$k^+(K) \colonequals \rk_2 \Cl^+(K) - \rk_2 \Cl(K),$$
and the \defi{$2$-adic isotropy rank} of $K$ is
$$k_4(K) \colonequals \rk_2 \ClfourK - \rk_2 \Cl(K).$$
\end{definition}

By Dummit--Voight \cite[Theorem 6.1]{DV}, we have
\begin{equation*}\label{eqn:isotropyranks}
\begin{aligned}
k^+(K) &= \dim_{\F_2} \im(\varphi_K) \cap V_\infty = \dim_{\F_2} \ker(\varphi_{K,2}) - \dim_{\F_2} \ker(\varphi_K) \\ 
k_4(K) &= \dim_{\F_2} \im(\varphi_K) \cap V_2 = \dim_{\F_2} \ker(\varphi_{K,\infty}) - \dim_{\F_2} \ker(\varphi_K)
\end{aligned}
\end{equation*} 
hence the nomenclature given in Definition \ref{defn:isotropyranks}. Moreover, there is a classical equality
\begin{equation} \label{eqn:k4kpc}
k_4(K) = k^+(K) + r_2
\end{equation}
(see for example, Theorem 2.2 of Lemmermeyer \cite{Lemmermeyer} and also Theorem \ref{thm:RDRP}  below).  

\subsection{Connections to $\Sel_2(K)$ via class field theory}

There is a natural, well-defined map $\Sel_2(K) \rightarrow \Cl(K)[2]$ sending $[\alpha] \mapsto [\frakaI]$ where $\frakaI^2 = (\alpha)$; this map is surjective and fits into the exact sequence 
    \begin{equation}\label{eq:2ses} 1 \rightarrow \cO_K^\times/(\cO_K^{\times})^2 \rightarrow \Sel_2(K) \rightarrow \Cl(K)[2] \rightarrow 1.
    \end{equation}
    
In addition, the $2$-Selmer signature map arises naturally in class field theory as follows.  Let $H \supseteq K$ be the Hilbert class field of $K$. Class field theory provides an isomorphism $\Gal(H\,|\,K) \simeq \Cl(K)$; let $H^{(2)}$ 
denote the fixed field of the subgroup $\Cl(K)^2$. The Kummer pairing
\begin{equation*}
    \begin{aligned}
    \Gal(H^{(2)}\,|\,K) \times \ker(\varphi_K) &\to \{\pm 1\} \\
    (\tau,[\alpha]) &\mapsto \frac{\tau(\sqrt{\alpha})}{\sqrt{\alpha}}
    \end{aligned}
\end{equation*}
is (well-defined and) perfect \cite[(3.11)]{DV}. The Artin reciprocity map provides a canonical isomorphism $\Gal(H^{(2)}\,|\,K) \simeq \Cl(K)/\!\Cl(K)^2$ and so we can rewrite the above map instead as  
\begin{equation} \label{eqn:Kummer}\begin{aligned}
 \Cl(K)/\!\Cl(K)^2  \times  \ker(\varphi_K) & \to \{\pm1\}. \\
 \end{aligned} 
 \end{equation}
 
 The pairing \eqref{eqn:Kummer} is the first of four perfect pairings \cite[Lemma 3.10]{DV} (see also Lemmermeyer \cite[Theorem 6.3]{Lemmermeyer}); the other three perfect pairings are  \begin{equation}\label{eqn:Kummer2}
 \begin{aligned}
 \ClfourK/\!\ClfourK^2  \times  \ker(\varphi_{K,\infty}) & \to \{\pm1\}, \\
 \Cl^+(K)/\!\Cl^+(K)^2  \times  \ker(\varphi_{K,2}) & \to \{\pm1\}, \mbox{ and} \\
 \ClfourKplus/\!\ClfourKplus^2  \times \Sel_2(K) & \to  \{\pm1\},
 \end {aligned}
 \end{equation}
where $\ClfourKplus$ denotes the ray class group of $K$ of  conductor $4\cdot\infty$.
 
\section{Galois module structures} \label{sec:galmodstruct}

We next study the Galois module structure on the arithmetic objects introduced in the previous section; we will continue in the next section with more precise results in the odd degree case.  Our results overlap substantially with those of Taylor \cite{Taylor}.

\smallskip

From now on, suppose that $K$ is Galois over $\Q$, with Galois group $G_K \colonequals \Gal(K\,|\,\Q)$.  We work throughout with left $\F_2[G_K]$-modules.  (We could consider more generally structures implied by the action of a nontrivial automorphism group $\Aut(K)$, and many of the results below could be generalized to this setting; we focus here on the extreme case, where $\Aut(K)$ is as large as possible.)

\subsection{Basic invariants}

We first prove Galois invariance of the signature spaces in generality.  Recall that a $\F_2$-bilinear form $b \colon V \times V \to \F_2$ on an $\F_2[G]$-module $V$ is \defi{$G$-invariant} if $b(\alpha,\beta) = b(\sigma(\alpha),\sigma(\beta))$ for all $\sigma \in G$.

\begin{proposition} \label{prop:GstructV2}
The following statements hold.
\begin{enumalph}
\item If $K$ is totally real, then $V_\infty(K) \simeq \F_2[G_K]$ as $\F_2[G_K]$-modules; otherwise, $V_\infty(K)$ is trivial.  In either case, the bilinear form $b_\infty$ is $G_K$-invariant.
\item We have $V_2(K) \simeq \F_2[G_K]$ as $\F_2[G_K]$-modules, and $b_2$ is $G_K$-invariant.
\end{enumalph}
\end{proposition}

\begin{proof}
We begin with (a), and suppose that $K$ is totally real.  The Galois group $G_K$ acts on $V_\infty(K)$ (on the left) via its permutation action on the (index) set of real places of $K$ (as $v \mapsto v \circ \sigma^{-1}$), so $V_\infty(K) \simeq \F_2[G_K]$ as $\F_2[G_K]$-modules.  Since $b_\infty$ is defined as the product over real places $v$, it is $G_K$-invariant.  

For (b), we follow the proof in Dummit--Voight \cite[Proposition 4.4]{DV}. The map $a \mapsto 1+2a$ induces an isomorphism
\begin{equation*}
\calO_{K,(2)}/2\calO_{K,(2)} \xrightarrow{\sim} \bigl(\calO_{K,(2)}^\times/(1+4\calO_{K,(2)})\bigr)[2] 
\end{equation*}
which is visibly $G_K$-equivariant.  By Lemma \ref{lem:genfact}, the right hand side is isomorphic to $V_2(K)$ as $\F_2[G_K]$-modules, since $(1+4\calO_{K,(2)})^2 \leq 1+8\calO_{K,(2)} \leq \calO_{K,(2)}^{\times 2}$.

Finally, we show $b_2$ is $G_K$-invariant.  Let $\alpha,\beta \in \cO_{K,(2)}^\times$ and let $v \mid 2$ be a prime of $K$.  
Since $G_K$ acts transitively (on the left) on the set of places $\{v : v \mid (2)\}$ with stabilizers $D_v \colonequals \Aut(K_v)$ the decomposition group, choosing a place $v$ we have
\[ b_2(\alpha,\beta) = \prod_{\tau D_v \in G_K/D_v} (\alpha,\beta)_{\tau(v)} \]
well defined.  
The Hilbert symbol $(\_, \_)_v$ is $G_K$-equivariant and $D_v$-invariant, so for $\sigma \in G_K$,
\[ b_2(\sigma(\alpha),\sigma(\beta)) = \prod_{\tau D_v \in G_K/D_v} (\sigma(\alpha),\sigma(\beta))_{\tau(v)} = \prod_{\tau} (\alpha,\beta)_{(\sigma^{-1}\tau)(v)} = \prod_{\tau} (\alpha,\beta)_{\tau(v)} = b_2(\alpha,\beta) \]
since $\sigma$ permutes the cosets of $D_v$ in $G_K$.
\end{proof}

\begin{lemma}  \label{lem:2selmequiv}
The $2$-Selmer signature map $\varphi_K$ is $G_K$-equivariant. 
\end{lemma}

\begin{proof} 
We show that both $\sgn_\infty$ and $\sgn_2$ are $G$-equivariant which implies that the induced maps $\varphi_{K,\infty}$ and $\varphi_{K,2}$ are $G$-equivariant  as well. For $\sgn_\infty$, we may suppose that $K$ is totally real, and then
\begin{equation*} 
\sgn_\infty(\sigma(\alpha))=(\sgn(v(\sigma(\alpha))))_v=(\sgn((\sigma^{-1}v)(\alpha)))_{v}=\sigma(\sgn_\infty(\alpha)) 
\end{equation*}
so $\sgn_\infty$ is $G_K$-equivariant. 

To show that $\sgn_2$ is $G_K$-equivariant, we observe that $\sgn_2$ is simply the composition of a natural embedding and projection.  
\end{proof}

\begin{corollary} \label{cor:Gkinv}
For a Galois number field $K$, the image of the $2$-Selmer signature map $\varphi_K$ is a $G_K$-invariant maximal totally isotropic subspace.
\end{corollary}

\begin{proof}
Combine Theorem \ref{Thm:MaximalIsotropic} with Proposition \ref{prop:GstructV2} and Lemma \ref{lem:2selmequiv}.
\end{proof}

\subsection{Duals and pairings}  \label{sec:dual_all}

We now treat some issues of duality, with an application to the Kummer pairing.  Let $G$ be a finite group and let $V$ be a finitely generated (left) $\F_2[G]$-module.

\begin{definition}\label{def:dual} The \defi{dual} of $V$ is the $\F_2$-vector space 
\begin{equation*} 
V\spcheck \colonequals \Hom_{\F_2}(V,\F_2)
\end{equation*}
equipped with the (left) $\F_2[G]$-action, arising from extending $\F_2$-linearly the natural $G$-action: if $\sigma \in G$, $f \in V\spcheck$, and $x \in V$, then $(\sigma f)(x) \colonequals f(\sigma^{-1} x)$.  We say $V$ is \defi{self-dual} if $V \simeq V\spcheck$ as $\F_2[G]$-modules.
\end{definition}  

The canonical evaluation pairing
\begin{equation} \label{eqn:natpairing}
\begin{aligned}
    e \colon V\spcheck \times V &\to \F_2 \\
    e(f,x) &= f(x)  
    \end{aligned}
\end{equation}
is nondegenerate and $G$-invariant, so gives a canonical isomorphism $V \xrightarrow{\sim} (V\spcheck)\spcheck$ as $\F_2[G]$-modules.   

\begin{lemma} \label{lem::AllKummerPairings}
For $K$ a Galois number field, the Kummer pairings \eqref{eqn:Kummer}--\eqref{eqn:Kummer2} induce canonical isomorphisms of $\F_2[G_K]$-modules:
 \begin{equation} \label{eqn:ClClK2} 
 \begin{array}{rcl}
 \Cl(K)/\!\Cl(K)^2  &\simeq& \ker(\varphi_K)\spcheck  \\
  \ClfourK/\!\ClfourK^2 & \simeq&  \ker(\varphi_{K,\infty})\spcheck  \\
 \Cl^+(K)/\!\Cl^+(K)^2 & \simeq &  \ker(\varphi_{K,2})\spcheck  \\
 \ClfourKplus/\!\ClfourKplus^2 & \simeq&  \Sel_2(K)\spcheck 
\end{array}
 \end{equation}
\end{lemma}

\begin{proof}
We work with the first line, the others follow by the same argument.
The Kummer isomorphism
\begin{equation} 
K^\times/K^{\times 2} \xrightarrow{\sim} \Hom(\Gal(\Kbar\,|\,K),\{\pm 1\}) 
\end{equation}
is $G_K$-equivariant and defines a canonical isomorphism $\ker(\varphi_K) \xrightarrow{\sim} \Hom(\Gal(Q\,|\,F),\{\pm 1\})$, where $Q$ is the maximal subfield whose Galois group has exponent dividing 2 in the Hilbert class field of $K$.  The Artin map defines a canonical $G_K$-equivariant isomorphism $\Gal(Q\,|\,F) \xrightarrow{\sim} \Cl(K)/\Cl(K)^2$.  Combining these with the evaluation map then gives a canonical pairing 
\begin{equation}
\Cl(K)/\Cl(K)^2 \times \ker(\varphi_K) \to \{\pm 1\}
\end{equation}
as claimed.  This pairing may be explicitly described as
\[ ([\fraka], [\alpha]) \mapsto \legen{\alpha}{\fraka} \]
where $\fraka \subseteq \calO_K$ is an ideal of odd norm, $\alpha \in \calO_K$ is coprime to $\fraka$, and $\displaystyle{\legen{\alpha}{\fraka}}$ is the Jacobi symbol.
\end{proof} 

Applying Lemma \ref{lem:genfact} to the groups on the left-hand side of \eqref{eqn:ClClK2} gives (now noncanonical) $\F_2[G_K]$-module isomorphisms $\Cl(K)[2] \simeq \ker(\varphi_K)\spcheck$, etc.

\begin{lemma} \label{lem:OrthoDecomp1} 
Let $b \colon V \times V \to \F_2$ be a $G$-invariant $\F_2$-bilinear form, and let $W,W' \subseteq V$ be irreducible $\F_2[G]$-modules.  If $W\spcheck \not\simeq W'$ as $\F_2[G]$-modules, then $b(W,W')=\{0\}$.  
\end{lemma}

\begin{proof}
Restricting $b$, we obtain an $\F_2[G]$-module map $W' \to W\spcheck$ by $w' \mapsto b(\_,w')$; by Schur's lemma, this map is either zero or an isomorphism, and the result follows.
\end{proof}

Lemma \ref{lem:OrthoDecomp1}, although easy to prove, is fundamental in what follows: it shows that when a decomposition of $V$ into irreducibles is possible, it is already \emph{almost} an orthogonal decomposition.  

To conclude this section, we refine this into a canonical orthogonal decomposition.  Suppose $G$ has \emph{odd} order, so the category of $\F_2[G]$-modules is semisimple.  Let $W$ be an irreducible $\F_2[G]$-module.  We write $V_W$ for the $W$-isotypic component of $V$ in a decomposition of $V$ into irreducibles.  Suppose that $V$ is equipped with a symmetric, $G$-invariant, $\F_2$-bilinear form.  Then by Lemma \ref{lem:OrthoDecomp1} we have a canonical decomposition as $\F_2[G]$-modules
\begin{equation} \label{eqn:candecomp}
V \simeq \bigboxplus_{W} \,(V_W + V_{W\spcheck}), 
\end{equation}
where the orthogonal direct sum is indexed by irreducibles $W$ up to isomorphism and duals.  We call the decomposition given in \eqref{eqn:candecomp} the \defi{canonical orthogonal decomposition} of $V$.

\section{Galois module structures for odd degree extensions} \label{sec:galmodstruct_odd}

In this section, we suppose throughout that $K$ has \emph{odd} degree (but remains Galois).  Then $K$ is totally real and the only roots of unity in $K$ are $\pm 1$.  Moreover, since $\#G_K$ is odd, the category of left $\F_2[G_K]$-modules is semisimple.  

\subsection{Basic invariants}

We quickly prove two standard lemmas, for completeness.  

\begin{lemma}\label{lem:units}  
We have $\cO_K^\times/(\cO_K^{\times})^2 \simeq \F_2[G_K]$ as $\F_2[G_K]$-modules.
\end{lemma}

\begin{proof} 
We consider $\calO_K^\times$ as a $\Z[G_K]$-module multiplicatively.
By Dirichlet's unit theorem,  
\[ (\calO_K^\times/\{\pm 1\} \otimes_\Z \R) \oplus \R \simeq \R[G_K] \]
as $\R[G_K]$-modules where $\R$ has trivial $G_K$ action (corresponding to the trace zero hyperplane in the Minkowski embedding).  Counting idempotents, we conclude that
		\begin{equation} \label{eqn:OKtimes2}
		    (\calO_K^\times/\{\pm 1\} \otimes_\Z \Z_{(2)}) \oplus \Z_{(2)} \simeq \bZ_{(2)}[G_K] 
		\end{equation}
		as $\Z_{(2)}$-modules; tensoring \eqref{eqn:OKtimes2} with $\Z/2\Z$ and using that $\{\pm 1\}$ has trivial action gives
\[ \calO_K^{\times}/(\calO_K^{\times})^2 \simeq \calO_K^{\times}/\{\pm 1\}(\calO_K^{\times})^2 \times \{\pm 1\} \simeq \F_2[G_K]. \qedhere \]
\end{proof}

\begin{corollary} \label{cor:sel2ck2}
We have $\Sel_2(K) \simeq \F_2[G_K] \oplus \Cl(K)[2]$ as $\F_2[G_K]$-modules.
\end{corollary}

\begin{proof}
Since $\F_2[G_K]$ is semisimple, the short exact sequence \eqref{eq:2ses} splits as $\F_2[G_K]$-modules; the result then follows from Lemma \ref{lem:units}.
\end{proof}

\begin{lemma}\label{lem:cl2}
For any odd Galois number field $K$, the $G_K$-invariant subspace of each of the $\F_2[G_K]$-modules $\Cl(K)[2]$, $\Cl^+(K)[2]$ and $\ClfourK[2]$ is trivial, whereas the $G_K$-invariant subspace of $\Cl(K)_{4}^+(K)$ is isomorphic to $\F_2$.
\end{lemma}

\begin{proof}
Let $C(K)$ denote one of the groups under consideration, and let $C(\Q)$ denote the ray class group of the same modulus but over $\Q$.  The norm map induces a group homomorphism $C(K)[2] \to C(\Q)[2]$, and on $G_K$-invariants it is an isomorphism, with inverse extension of ideals, since $n$ is odd.  Indeed, if $[\fraka] \in C(K)[2]$ is $G_K$-invariant, then $[\Nm(\fraka)]=[\fraka]^n=[\fraka] \in C(K)[2]$; similarly, if $[(a)] \in C(\Q)[2]$ then $[\Nm(a\Z_K)]=[(a)]^n=[(a)] \in C(\Q)[2]$.  Since the groups $\Cl(\Q),\Cl^+(\Q),\Cl_4(\Q)$ are trivial and $\Cl_4^+(\Q) \simeq \Z/2\Z$, the result follows. 
\end{proof}

\subsection{First reflection principle} \label{sec:reflGK}

In this section, we show that the Galois module structure of the $2$-Selmer group and Kummer duality imply rank inequalities on the class group, classically known as a \emph{reflection theorem}.  Let $W$ be an irreducible (left) $\F_2[G_K]$-module, and for a finitely generated $\Z[G_K]$-module $M$, let $\rk_W(M) \in \Z_{\geq 0}$ be the multiplicity of $W$ in a decomposition of $M/2M$ into irreducible $\F_2[G_K]$-modules.  We recall Lemma \ref{lem:genfact}, which gives an isomorphism $M/2M \simeq M[2]$ for a torsion, finitely generated $\Z(2)[GK]$-module $M$, in particular giving $\rk_W(M)=\rk_W(M[2])$.  

As mentioned in the introduction, our reflection theorems (Proposition \ref{prop:multNGKFS}, Proposition \ref{prop:F2GK}, and Theorem \ref{thm:RDRP}) are special cases of the very general $T$-$S$-reflection theorem of Gras \cite[Th\'eor\`eme 5.18]{Gras}.  Our goal in the next few sections is to give a direct proof of these results: it shows that they can be read off from the $2$-Selmer group, i.e., that they are intrinsic to the underlying structure of the image of the $2$-Selmer group, as we will see below.

We use the notation
\begin{equation}  \label{eqn:rhowsK}
\begin{aligned}
\rho_W(K) &\colonequals \rk_W \Cl(K) \\
\rho_W^+(K) &\colonequals \rk_W \Cl^+(K) \\
\rho_{4,W}(K) &\colonequals \rk_W \Cl_4(K).
\end{aligned}
\end{equation}

\begin{proposition} \label{prop:multNGKFS}
Let $K$ be a Galois number field of odd degree, and let $W$ be an irreducible $\F_2[G_K]$-module.   Then
\begin{equation} \label{eqn:rhowk}
\rho_{W}(K) - \rho_{W\spcheck}(K) = \rk_W \F_2[G_K] - \rk_{W\spcheck} S(K)
\end{equation}
and
\begin{equation} \label{eqn:rkcldk}
\abs{\rho_W(K) - \rho_{W\spcheck}(K)} \leq \rk_W \F_2[G_K].
\end{equation}
\end{proposition}

\begin{proof}
Since the short exact sequence in \eqref{eq:2ses} splits as a sequence of $\F_2[G_K]$-modules, decomposing $\Sel_2(K)$ under $\varphi_K$ gives
\begin{equation}\label{eqs::decompositions}
\cO_K^\times/(\cO_K^\times)^2  \oplus \Cl(K)[2] \simeq \Sel_2(K) \simeq S(K) \oplus \ker(\varphi_K)
\end{equation}
as $\F_2[G_K]$-modules.  By Lemma \ref{lem:units}, we have $\cO_K^\times/(\cO_K^\times)^2 \simeq \F_2[G_K]$.  By Lemma \ref{lem::AllKummerPairings}, we have $\Cl(K)[2] \simeq \ker(\varphi_K)\spcheck$, so
$\rho_W(K) = \rk_{W\spcheck} \ker(\varphi_K)$.  Plugging these in and taking $\rk_{W\spcheck}$ in \eqref{eqs::decompositions} yields
\begin{equation}  \label{eqn:mWwWK}
 \rk_W \F_2[G_K] + \rho_{W\spcheck}(K) = \rk_{W\spcheck} S(K)  + \rho_{W}(K)
\end{equation}
so 
\[ \rho_{W}(K) - \rho_{W\spcheck}(K) = \rk_W \F_2[G_K] - \rk_{W\spcheck} S(K) \leq \rk_W \F_2[G_K] \] 
giving \eqref{eqn:rhowk}.  Repeating the argument with $W\spcheck$, noting $\rk_W \F_2[G_K]=\rk_{W\spcheck} \F_2[G_K]$, and negating then gives \eqref{eqn:rkcldk}.
\end{proof}

In particular, we see from the proof of Proposition \ref{prop:multNGKFS} that the inequality is refined by the equality \ref{eqn:rhowk}, with the discrepancy in the inequality being measured by the group $S(K)$.  This is the simplest instance of the motivation of our paper: we seek to understand structural properties of the $2$-Selmer signature map, from which reflection principles are corollaries.

\subsection{Isotropy ranks}\label{subsec:S(K)}

Similar inequalities govern the narrow class group and its relationship to the class group, encoded in the $2$-Selmer group.  To measure these contributions, we make the following definitions.  Throughout, let $W$ be an irreducible $\F_2[G_K]$-module.  

\begin{definition}\label{def:chirank}
The \defi{archimedean $W$-isotropy rank} of $K$ is
\[ k_W^+(K) \colonequals \rk_W \Cl^+(K) - \rk_W \Cl(K) \]
and the \defi{$2$-adic $W$-isotropy rank} of $K$ is
$$k_{4,W}(K) \colonequals \rk_W\Cl_4(K) - \rk_W\Cl(K).$$
\end{definition}

We have $k_W^+(K),k_{4,W}(K) \in \Z_{\geq 0}$, since $\Cl^+(K),\Cl_4(K)$ surject onto $\Cl(K)$.

\begin{proposition}  \label{prop:kchi4SK}
We have
\begin{equation*}
\begin{aligned}
\Cl^+(K)[2] &\simeq \Cl(K)[2] \oplus (S(K) \cap V_\infty(K))\spcheck, \text{ and } \\
\Cl_4(K)[2] &\simeq \Cl(K)[2] \oplus (S(K) \cap V_2(K))\spcheck.
\end{aligned}
\end{equation*}
In particular,
\begin{equation*}
\begin{aligned}
k_W^+(K) &= \rk_{W\spcheck}(S(K) \cap V_\infty(K)), \text{ and } \\
k_{4,W}(K) &=\rk_{W\spcheck}(S(K) \cap V_2(K)). 
\end{aligned}
\end{equation*}
\end{proposition}

\begin{proof}
We have that $S(K) \cap V_\infty(K) \simeq \ker(\varphi_{K,2})/\ker(\varphi_K)$; since $\ker(\varphi_{K,2})$ and $\ker(\varphi_K)$ are \emph{Kummer dual} to $\Cl^+(K)[2]$ and $\Cl(K)[2]$ by Lemma \ref{lem::AllKummerPairings}, the first isomorphism follows; taking $W$-rank and subtracting gives
\[\rk_{W\spcheck}(S(K) \cap V_\infty(K)) = \rk_{W} \Cl^+(K) - \rk_{W} \Cl(K). \]
The second isomorphism and equality follow similarly.
\end{proof} 

A further duality is reflected in the totally positive elements in the $2$-Selmer group, as follows.

\begin{theorem} \label{thm:RDRP} 
Let $K$ be a Galois number field of odd degree.  Then
\[ \Cl^+(K)[2] \simeq \ClfourK[2]\spcheck \]
as $\F_2[G_K]$-modules.
\end{theorem}

\begin{proof}
Let $\Sel_2^+(K) \colonequals \ker(\varphi_{K,\infty})$ be the classes in the $2$-Selmer group represented by a totally positive element; then $\Sel_2^+(K) \simeq \ClfourK[2]\spcheck$ by Lemma \ref{lem::AllKummerPairings}; we show $\Sel_2^+(K) \simeq \Cl^+(K)[2]$ as $\F_2[G_K$]-modules.

Our proof considers the analogue for $\Sel_2^+(K)$ of the exact sequence \eqref{eq:2ses}.  Let $P_K$ be the group of principal fractional ideals of $K$, and let $P_{K,>0}$ be the subgroup of $P_K$ consisting of principal fractional ideals generated by a totally positive element.  The map $K^\times \to P_K$ sending $\alpha \mapsto (\alpha)$ is surjective and $G_K$-equivariant with kernel $\calO_K^\times$; it induces the exact sequence
\begin{equation*} 
1 \to \calO_K^\times/\calO_{K,>0}^{\times} \to K^\times/K_{>0}^{\times} \to P_K/P_{K,>0} \to 1.
\end{equation*}
By weak approximation, the natural map $K^\times/K_{>0}^\times \to V_\infty(K)$ is a $G_K$-equivariant isomorphism, and so $K^\times/K_{>0}^{\times} \simeq \F_2[G_K]$ by Proposition \ref{prop:GstructV2}(a).  Therefore we obtain a  isomorphism of $\F_2[G_K]$-modules
\begin{equation} \label{eqn:F2G}
\F_2[G_K]   \simeq (\cO_{K}^\times/ \cO_{K,>0}^{\times} ) \oplus (P_K/P_{K,>0} ).
\end{equation}

The natural $G_K$-equivariant map $P_K \to \Cl^+(K)$ defined by $(\alpha) \mapsto [(\alpha)]$ has kernel $P_{K,>0}$ 
and so we have a canonical injection $P_K/P_{K,>0} \hookrightarrow \Cl^+(K)$. Since $P_K^2$ is a subgroup of $P_{K,>0}$ the image of the injection is contained in $\Cl^+(K)[2]$.  Therefore, the map 
\begin{equation*}
\begin{aligned}
\Sel_2^+(K) &\to \Cl^+(K)[2]/(P_K/P_{K,>0})
\end{aligned}
\end{equation*}
mapping the class of $\alpha \in K^\times$ to the class of the fractional ideal $\frakaI$ such that $\frakaI^2=(\alpha)$ is well-defined; it is also visibly surjective, and so fits into the short exact sequence
\begin{equation*}\label{eq:2ses0} 1 \rightarrow \cO_{K,>0}^\times/\cO_{K}^{\times 2} \rightarrow \Sel_2^+(K) \rightarrow \Cl^+(K)[2]/(P_K/P_{K,>0} )\rightarrow 1
    \end{equation*}
    of $\F_2[G_K]$-modules, giving the isomorphism
    \begin{equation} \label{eqn:PKClk2} 
    P_K/P_{K,>0} \oplus \Sel_2^+(K) \simeq \calO_{K,>0}^{\times}/\calO_{K}^{\times 2} \oplus \Cl^+(K)[2]. 
    \end{equation}
    Adding $\calO_{K}^{\times}/\calO_{K,>0}^{\times}$ to both sides of \eqref{eqn:PKClk2}, and using \eqref{eqn:F2G} and Lemma \ref{lem:units} we conclude 
    \[ \F_2[G_K] \oplus \Sel_2^+(K) \simeq \F_2[G_K] \oplus \Cl^+(K)[2] \] 
    and cancelling gives the result.  
\end{proof}

By semisimplicity, we can decompose
\[ S_K \simeq (S_K \cap V_\infty) \oplus (S_K \cap V_2) \oplus S_K' \]
as $\F_2[G_K]$-modules for some choice $S_K' \subseteq S_K$, well-defined up to isomorphism.  We call $S_K'$ a \defi{coordinate complement} to $S_K$ in $V$.  With this notation, we immediately turn to our next reflection principle: again, all we use is $\F_2[G_K]$-module structure and Kummer duality.

\begin{proposition} \label{prop:F2GK}
Let $W$ be an irreducible $\F_2[G_K]$-module, and let $S_K'$ be a coordinate complement to $S_K$ in $V$. 
Then
\begin{equation} \label{eqn:whorplus}
\rho_W^+(K) - \rho_{W\spcheck}^+(K) = \rk_W \F_2[G_K] - \rk_W(S \cap V_\infty) - \rk_{W\spcheck}(S \cap V_2) - \rk_{W\spcheck} S'_K
\end{equation}
and
\begin{equation} \label{eqn:rkcldk_plus}
\begin{aligned}
\abs{\rho_W^+(K) - \rho_{W\spcheck}^+(K)}=\abs{\rho_{4,W}(K) - \rho_{4,W\spcheck}(K)} &\leq \rk_W \F_2[G_K].
\end{aligned}
\end{equation}
\end{proposition}

\begin{proof}
As $K$ is fixed, we drop it from the notation.  By Proposition \ref{prop:kchi4SK}, we obtain
\begin{equation}
\rk_{W\spcheck} S = k_W^+ + k_{4,W} + \rk_{W\spcheck} S'.
\end{equation}
From \eqref{eqn:mWwWK} we get 
\[ \rk_{W\spcheck} S = \rk_W \F_2[G_K] + \rho_{W\spcheck} - \rho_W, \]
so plugging and rearranging gives
\begin{equation}
\begin{aligned}
\rho_W^+ = k_W^+ + \rho_W &= \rk_W \F_2[G_K] + \rho_{W\spcheck} - k_{4,W} - \rk_{W\spcheck} S' \\
\rho_W^+ - \rho_{W\spcheck}^+ &= \rk_W \F_2[G_K] - k_{W\spcheck}^+ - k_{4,W} - \rk_{W\spcheck} S' \leq \rk_W \F_2[G_K].
\end{aligned}
\end{equation}
Repeating with $W$ replaced by $W\spcheck$ gives the inequality $\abs{\rho_W^+(K) - \rho_{W\spcheck}^+(K)} \leq \rk_W \F_2[G_K]$ in \eqref{eqn:rkcldk_plus}.  By Theorem \ref{thm:RDRP}, we have $\rho_W^+(K)=\rho_{4,W\spcheck}(K)$ and $\rho_{W\spcheck}^+(K)=\rho_{4,W}(K)$, which gives the equality in \eqref{eqn:rkcldk_plus} and finishes the proof.
\end{proof}

\begin{proposition} \label{prop:kwKW4}
Let $W$ be an irreducible $\F_2[G_K]$-module, and let $S_K'$ be a coordinate complement to $S_K$ in $V$. 
Then
\begin{equation} \label{eqn:kwkwps}
k_W^+(K) + k_{W\spcheck}^+(K) = \rk_W \F_2[G_K] - \rk_{W\spcheck} S'_K.
\end{equation}
Moreover, $S'_K$ is self-dual and 
\begin{equation} \label{eqn:k0d0d}
0 \leq k_W^+(K) + k_{W\spcheck}^+(K) = k_{4,W}(K) + k_{4,W\spcheck}(K) \leq \rk_W \F_2[G_K]. 
\end{equation}
\end{proposition}

\begin{proof}
We again drop $K$ from the notation.  For the equality in \eqref{eqn:k0d0d}, by Theorem \ref{thm:RDRP}, we have 
\[ \rho_W^+ + \rho_{W\spcheck}^+ = \rho_{4,W} + \rho_{4,W\spcheck}; \]
subtracting $\rho_W + \rho_{W\spcheck}$ from both sides gives the result.  From \eqref{eqn:whorplus} we have
\[ k_W^+ + k_{4,W} + \rho_W - \rho_{W\spcheck} = \rk_W \F_2[G_K] - \rk_{W\spcheck} S' \]
But $k_{4,W} + \rho_W = \rho_{4,W} = \rho_{W\spcheck}^+$ by Theorem \ref{thm:RDRP}, so
\[ k_W^+ + k_{W\spcheck}^+ = \rk_W \F_2[G_K] - \rk_{W\spcheck} S' \leq \rk_W \F_2[G_K] \]
giving \eqref{eqn:kwkwps}.  To restore symmetry, we repeat the same argument with $W\spcheck$ and conclude that $\rk_{W} S' = \rk_{W\spcheck} S'$, so in fact $S'$ is self-dual.
\end{proof}

Just as in Proposition \ref{prop:multNGKFS}, we see from the proof of Proposition \eqref{prop:kwKW4} that the real content lies in the equality \eqref{eqn:kwkwps}, i.e., the discrepancy in the upper bound \eqref{eqn:k0d0d} is measured by the (noncanonically defined) ``diagonal subspace'' $S'(K) \subseteq S(K)$.

We deduce corollaries of these statements in the abelian case in section \ref{sec:conseq}.

\section{Galois module structures for odd degree abelian extensions} \label{sec:Galois}

In this section, we specialize further and suppose that the odd order group $G$ is \emph{abelian} and prove the main structural results of the paper, first for class groups and then for unit signatures. 

\subsection{Duality in the abelian case}

We begin by revising notation and duality in the abelian setting.  Let $\Ftwobar$ be a (fixed) algebraic closure of $\F_2$.  An \defi{$\Ftwobar$-character} of $G$ is a group homomorphism $\chi \colon G \to \Ftwobar^\times$.  For an $\Ftwobar$-character $\chi$, let $\F_2(\chi) \subseteq \Ftwobar$ be the subfield generated by the values of $\chi$.  Then $\F_2(\chi)$ is a finite extension of $\F_2$: more precisely, if $\chi$ has (odd) order $d$ and $2$ has order $f$ in $(\bZ/d\bZ)^\times$, then $\F_2(\chi)\simeq\F_{2^f}$ as $\F_2$-vector spaces.  
The group $G$ acts naturally on $\F_2(\chi)$ via multiplication by $\chi(\sigma)$; thus $\F_2(\chi)$ is a cyclic, irreducible $\F_2[G]$-module, generated by $1$.  Conversely, choosing a cyclic generator, every irreducible $\F_2[G]$-module is of the form $\F_2(\chi)$ for some $\Ftwobar$-character $\chi$.  

By character theory, two such modules $\F_2(\chi)$ and $\F_2(\chi')$ are isomorphic if and only if there exists $\psi \in \Gal(\Ftwobar\,|\,\F_2)$ such that $\chi'=\psi \circ\chi$.    
In particular, since $\psi$ is a power of the Frobenius automorphism, $\F_2(\chi) \simeq \F_2(\chi')$ if and only if $\chi'=\chi^{2^k}$ for some $k \in \Z$.  Moreover, $\Aut_{\F_2[G]}(\F_2(\chi)) \simeq \F_2(\chi)^\times$.  

There is also a simple way to understand duality when $G$ is abelian.  Let $V$ be a finitely generated $\F_2[G]$-module.  The map $\sigma \mapsto \sigma^{-1}$ for $\sigma \in G$ extends by $\F_2$-linearity to an involution $\phantom{i}\spdual \colon \F_2[G] \to \F_2[G]$, which is a ring automorphism when $G$ is abelian. We define $V\spdual$ (the contragredient representation) to be the $\F_2[G]$-module with the same underlying $\F_2$-vector space $V$ but with the action of $\F_2[G]$ under pullback from the involution map.  Explicitly, if $\gamma \in \F_2[G]$ and $x\spdual \in V\spdual$ denotes the same element $x \in V$ then $\gamma(x\spdual) \colonequals (\gamma\spdual(x))\spdual;$
in particular,  for $\sigma \in G$, then $\sigma(x\spdual) = \sigma\spdual(x)\spdual = \sigma^{-1}(x)\spdual$.  We conclude that $$\F_2(\chi)\spdual \simeq \F_2(\chi^{-1}) $$ 
as $\F_2[G]$-modules, which explains the notation $\chidual = \chi^{-1}$ from the introduction.  

\begin{remark}
Without the hypothesis that $G$ is abelian, starting with a left $\F_2[G]$-module $V$, we would obtain a \emph{right} $\F_2[G]$-module $V\spdual$.
\end{remark}

\begin{lemma} \label{lem:homisom}
There is a (non-canonical) $\F_2[G]$-module isomorphism $V\spdual \xrightarrow{\sim} V\spcheck$.
\end{lemma}

\begin{proof} 
Decomposing into irreducibles up to isomorphism, we may suppose without loss of generality that $V=\F_2(\chi)$.  We will compose the $\F_2$-linear maps $\spdual \colon V\spdual \to V$ with the natural map $V \to V\spcheck$ by trace to obtain an $\F_2[G]$-linear map as follows.  Consider the map
\begin{equation*}
\begin{aligned}
V\spdual \ &\to  \ V\spcheck \\
x\spdual \  &\mapsto  (y \mapsto \Tr(x\cdot y))
\end{aligned}
\end{equation*}
where $\Tr\colon \F_2(\chi) \to \F_2$ is the trace map.  This map is nonzero and $\F_2$-linear, and it is also $G$-equivariant because for any $y \in V$,
\begin{equation}
\begin{aligned}
\Tr( (\sigma(x^*))^* \cdot y) &= \Tr(\sigma^{-1} (x)\cdot y)=\Tr(\chi(\sigma^{-1}) x \cdot y)=\Tr(x \cdot \chi(\sigma^{-1})y) \\
&=\Tr(x\cdot \sigma^{-1}(y))=(\sigma \Tr)(x\cdot y) 
\end{aligned}
\end{equation}
for all $\sigma \in G$, $x^* \in V^*$, $y \in V$; hence, it is an isomorphism by Schur's lemma.
\end{proof}

\begin{remark}
See also Theorem \ref{Thm:OrthoDecomp} below, where we revisit the trace pairing on $\F_2[G]$ with its involution $\spdual$.
\end{remark}

\begin{lemma}\label{lem:self-duality}
Let $m \in \Z_{\geq 1}$ denote the (odd) exponent of the abelian group $G$. Every irreducible $\F_2[G]$-module is self-dual if and only if there exists $t \in \Z$ such that $2^t \equiv -1 \pmod{m}$, where $m$ is the exponent of $G$.
\end{lemma}

\begin{proof}
Let $\F_2(\chi)$ be an irreducible $\F_2[G]$-module, and let $d$ be the order of $\chi$.  We have $\F_2(\chi) \simeq \F_2(\chi)^*=\F_2(\chidual)$ if and only if $\chidual=\chi^{2^k}$ for some $k \in \Z$, i.e., $2^k \equiv -1 \pmod{d}$.  Choosing a character with order $d=m$ then gives the result.
\end{proof}

\begin{example}
The smallest (odd) values of $m \in \Z_{>0}$ where $-1 \not\in \langle 2 \rangle \leq (\Z/m\Z)^\times$ are $m=7$, $15$, $21$, and $23$.
\end{example}

\begin{example}\label{ex:prime}
Suppose $\#G = \ell$ is prime and let $f$ be the order of 2 in $(\bZ/\ell\bZ)^\times$.  Then there are $\frac{\ell - 1}{f}$ distinct, nontrivial $\F_2[G]$-modules, up to isomorphism. 
They are all isomorphic as $\F_2$-vector spaces to $\F_{2^{f}}$, a generator of $G$ acts by multiplication by a primitive $\ell$th root of unity $\zeta \in \F_{2^f}$, and two such are isomorphic if and only if $\zeta'=\zeta^{2^k}$ for some $k \in \Z$. 
Finally, all such modules are self-dual if and only if $f$ is even.
\end{example}

The Galois module structure has concrete implications for ranks.

\begin{example}
Let $K$ be a cyclic number field of odd prime degree $\ell$, and let $f$ denote the order of $2$ modulo $\ell$. Then taking $G=G_K$, applying the decomposition into irreducibles given in Example \ref{ex:prime} and Lemma \ref{lem:cl2}, we conclude that $f$ divides each of $\rk_2\Cl(K)$, $\rk_2\Cl^+(K)$, and $\rk_2 \ClfourK$, and that $\rk_2 \Cl^+_4 (K) \equiv 1 \pmod f$.
\end{example}

\subsection{Bilinear forms} 

For an $\Ftwobar$-character $\chi$ of $G$, we write $V_\chi$ for the $\F_2(\chi)$-isotypic component of $V$ and $V_{\chi^{\pm}} \colonequals V_\chi + V_{\chi^{-1}} = V_\chi + V_{\chidual}$ for the sum of the $\F_2(\chi)$- and $\F_2(\chidual)$-isotypic components of $V$.  (If $\chi$ is not self-dual, then this sum is direct; if $\chi$ is self-dual, then $V_{\chi^{\pm}}=V_\chi$.)  

\begin{example}\label{ex:mult1}
Since $G$ is abelian, each isomorphism class of irreducible $\F_2[G]$-modules occurs in $\F_2[G]$ with multiplicity $1$.  Therefore
\[ \F_2[G]_{\chi^{\pm}} \simeq
\begin{cases}
\F_2(\chi), & \text{ if $\chi$ is self-dual;} \\
\F_2(\chi) \oplus \F_2(\chidual), & \text{ otherwise.} 
\end{cases} \]
\end{example}

If $V$ is equipped with a symmetric, $G$-invariant, $\F_2$-bilinear form, then $V$ has a canonical orthogonal decomposition \eqref{eqn:candecomp}
\begin{equation} \label{eqn:candecomp0}
V \simeq \bigboxplus_{\chi} V_{\chi^{\pm}}, 
\end{equation}
where the orthogonal direct sum is indexed by characters $\chi$ taken up to isomorphism and inverses.  Consequently, it is enough to understand bilinear forms on the components $V_{\chi^{\pm}}$.  We use this technique to prove the following fundamental classification result.

\begin{theorem} \label{Thm:OrthoDecomp}
Let $G$ be an abelian group of odd order.  Then there is a unique $G$-invariant, symmetric, nondegenerate, $\F_2$-bilinear form on $\F_2[G]$ up to $G$-equivariant isometry, given by
\begin{equation} \label{eqn:bF2G}
\begin{aligned}
b \colon \F_2[G] \times \F_2[G] & \to \F_2 \\
(x,y) &\mapsto \Tr(x^* y).
\end{aligned}
\end{equation}
In the standard basis $g \in G$ for $\F_2[G]$, the form $b$ is the standard `dot product'.
\end{theorem}

\begin{proof}
Recalling that $\F_2[G]$ is commutative, we see that the trace pairing is symmetric because 
\[ b(y,x)=\Tr(y^* x)=\Tr((y^* x)^*) = \Tr(x^* y) = b(x,y) \]
and $G$-invariant because
\begin{equation*} 
b(\sigma x,\sigma y)=\Tr((\sigma x)^* \sigma y)=\Tr(\sigma^{-1}(x^*)\sigma y) = \Tr(x^* y)=b(x,y). \end{equation*}
Note ${}^*$ is the adjoint with respect to $b$, since 
\begin{equation}  \label{eqn:adjbnu}
b(\nu x,y) = \Tr((\nu x)^* y) = \Tr( \nu^* x^* y)=\Tr(x^* (\nu^* y)) = b(x,\nu^* y). 
\end{equation}
For all $\sigma \in G$, we have $\Tr(\sigma)=1$ if $\sigma=1$, else $\Tr(\sigma)=0$: indeed, in the standard basis, $\sigma$ acts on $\F_2[G]$ by a permutation matrix with no fixed points.  Thus $b(\sigma,\sigma)=\Tr(\sigma^{-1}\sigma)=1$ and $b(\sigma,\tau)=0$ for all $\sigma,\tau\in \F_2[G]$.  Thus in this basis, $b$ is indeed the `dot product', so is (symmetric and) nondegenerate.  

Let $b'$ be another $G$-invariant, symmetric, nondegenerate, $\F_2$-bilinear form on $\F_2[G]$.  We obtain an $\F_2[G]$-module isomorphism $\psi \colon \F_2[G] \xrightarrow{\sim} \F_2[G]\spcheck$ by $\psi(x)(y)=b(x,y)$, and similarly $\psi'$ with $b'$.  Let $\nu = \psi^{-1}(\psi'(1))$.  Then $b'(x,y)=b(\nu x,y)$ for all $x,y \in \F_2[G]$.  

Since $b,b'$ are symmetric, by \eqref{eqn:adjbnu} we have
\[ b(\nu x,y) = b'(x,y)=b'(y,x)=b(\nu y,x)=b(x,\nu y)=b(\nu^* x,y) \]
for all $x,y \in \F_2[G]$.  Taking $x=1$, by nondegeneracy we conclude that $\nu=\nu^*$.  
Since $\F_2[G]$ is a product of fields of characteristic $2$, the group of units $\F_2[G]^\times$ is an abelian group of odd order.  The subgroup of units fixed under the ring automorphism ${}^*$ is again of odd order, so squaring is an automorphism and thus $\nu = \mu^2$ for some $\mu \in \F_2[G]^\times$ with $\mu=\mu^*$.  Thus $\nu=\mu^* \mu$.  Since
\[ b'(x,y)=b(\nu x,y)=b(\mu^* \mu x,y) = b(\mu x, \mu y) \]
and therefore the $\F_2[G]$-module automorphism $x \mapsto \mu x$ defines a $G$-equivariant isometry from $b$ to $b'$.  
\end{proof}

\begin{example} \label{exm:Fq2g}
As this will be central to our investigation, we write down explicitly the pairing in Theorem \ref{Thm:OrthoDecomp} restricted to  orthogonal components as in \eqref{eqn:candecomp}.

If $\chi$ is self-dual, then $\F_2[G]_{\chi^{\pm}} \simeq \F_2(\chi) \simeq \F_{2^f}$ and the bilinear form is non-alternating when $\chi$ is trivial and alternating when $\chi$ is non-trivial. When $\chi$ is trivial then $\F_2(\chi)= \F_2$ and $b(1,1) = 1$ so the form is non-alternating. Now suppose that $\chi$ has order $d > 1$ and let $\zeta \in \F_2(\chi)$ be a primitive $d$th root of unity.  Since $\Tr_{\F_2(\chi)|\F_2}(1)=0$ (as $f$, the order of $2$ in $(\Z/d\Z)^\times$ is even), we have $b(\zeta^{k},\zeta^{k})=\Tr_{\F_2(\chi)|\F_2}(1)=0$ for all $k$, so by linearity we conclude that $b$ is alternating. 

If $\chi$ is not self-dual, then $\F_2[G]_{\chi^{\pm}} \simeq \F_2(\chi) \oplus \F_2(\chidual) \simeq (\F_{2^f})^2$ and the bilinear form is a sum of hyperbolic planes, pairing dual basis elements nontrivially.  Put another way, the canonical pairing \eqref{eqn:natpairing} induces a natural pairing on $\F_2(\chi)\spcheck \oplus \F_2(\chi)$, which can be described explicitly as
\begin{equation*} 
b((f,x),(g,y))=f(y)+g(x).
\end{equation*}
\end{example}

\subsection{Maximal totally isotropic subspaces}

In this section, we will classify the maximal isotropic subspaces of $\F_2[G] \boxplus \F_2[G]$ and study their isometry groups.  

We continue our hypothesis that $G$ is a finite abelian group of odd order.  Let $V$ be a finitely generated $\F_2[G]$-module equipped with a $G$-invariant, symmetric, $\F_2$-bilinear form.  We let $\Isom_G(V) \leq \Aut_G(V)$ be the group of $G$-equivariant isometries of $V$, i.e., the subset of $\F_2[G]$-module automorphisms of $V$ which preserve the bilinear form. 

\begin{lemma} \label{lem::isometries} 
Equip $\F_2[G]$ with the trace bilinear form $b$ \eqref{eqn:bF2G}. Let $\chi$ be an $\Ftwobar$-character of $G$ and let $q = \#\F_2(\chi)$. Then 
\[ \# \Isom_G (\F_2[G]_{\chi^\pm}) = 
 \begin{cases}  
 \sqrt{q} +1,   &  \text{  if $\chi$ is self-dual;} \\ 
  q-1,  &  \text{ otherwise.}  
  \end{cases} \]
\end{lemma}

\begin{proof}
First, suppose $\chi$ is self-dual, so $\F_2[G]_{\chi^\pm} \simeq \F_2(\chi)$.  As in the proof of Theorem \ref{Thm:OrthoDecomp}, the group $\Aut_G(V)$ of $\F_2[G]$-module automorphisms of $\F_2(\chi)$ are given by multiplication by an element $\nu \in \F_2(\chi)^\times$.  The subgroup $\Isom_G(V) \leq \Aut_G(V)$ of isometries are those for which 
\begin{equation}  \label{eqn:bxynu}
b(x, y) = b(\nu x, \nu y) = b( x, \nu \nu^* y)
\end{equation}
for all $x,y \in \F_2(\chi)$; since $b$ is nondegenerate, this is equivalent to $\nu\nu^*=1$.  The map $\nu \mapsto \nu\nu^*$ is the norm to the unique subfield of $\F_2(\chi)$ of index $2$; since the norm is surjective, we conclude that $\Isom_G(\F_2(\chi))$ is a cyclic group of cardinality $(q-1)/(\sqrt{q}-1)= \sqrt{q}+1$.  

Second, suppose $\chi$ is not self-dual, and write $V \colonequals \F_2[G]_{\chi^\pm} = \F_2(\chi) \oplus \F_2(\chi^*)$.  Then ${}^*$ acts on $V$ by $(x,y)^* = (y^*,x^*)$.  The group $\Aut_G(V)$ is given by coordinate-wise multiplication by $(\mu,\nu) \in \F_2(\chi)^\times \times \F_2(\chi^*)^\times$.  Such an automorphism is an isometry if and only if
\begin{equation} \label{eqn:x1x2y1y2numu}
\begin{aligned}
b(( x_1, x_2),( y_1,  y_2))  &=  b((\nu x_1, \mu x_2),( \nu y_1, \mu  y_2)) =  b((x_1, x_2),( \nu \mu^* y_1, \mu \nu^* y_2)) \\
& =  b((x_1, x_2),( \nu \mu^* y_1, (\nu \mu^*)^* y_2))  
\end{aligned}
\end{equation}
for all $(x_1,x_2), (y_1,y_2) \in \F_2(\chi) \times \F_2(\chi^*)$. The same nondegeneracy argument in the previous paragraph shows this is equivalent to $\nu \mu^*  = 1 \in \F_2(\chi)^\times$ (equivalently,  $\mu\nu^* = 1 \in \F_2(\chi^*)^\times$).  We conclude that $\Isom_G(V) \leq \Aut_G(V)$ consists of the elements $(\nu, (\nu^{-1})^*)$ with $\nu \in \F_2(\chi)^\times$, a cyclic group of cardinality $q-1$.
\end{proof}

\begin{lemma} \label{lem::isom=maxtotiso} 
Suppose the bilinear form on $V$ is nondegenerate.  Then all $G$-invariant maximal totally isotropic subspaces $S \subseteq V \boxplus V$ such that $S \cap (V \boxplus \{0\}) = S \cap (\{0\} \boxplus V)=\{0\}$ are in the same $G$-equivariant isometry class, and there are exactly $\#\Isom_G(V)$ of them.  
\end{lemma}

\begin{proof} 
See Dummit--Voight \cite[Lemma A.9]{DV} for a proof in the case of $\F_2$-vector spaces; the method of proof gives the same result for $\F_2[G]$-modules. For example, the isometry $\tau$ in \cite[Lemma A.9]{DV}  is automatically a $G$-equivariant isomorphism when $S$ is $G$-invariant: the element $g \circ (v + \tau(v)) = g \circ v + g \circ (\tau(v))$ is again an element of $S$, so $\tau(g \circ v) = g \circ (\tau(v))$
\end{proof}

In our main classification in section \ref{sec:conseq}, we will need to classify Galois invariant maximal totally subspaces in the setting of the following theorem.

\begin{theorem} \label{thm:maxisoselfdual} 
Let $G$ be an abelian group of odd order, and let $\chi$ be an $\Ftwobar$-character of $G$ and let $q \colonequals \#\F_2(\chi)$.  Let $V \colonequals V_1 \boxplus V_2$ with each $V_i \colonequals \F_2[G]_{\chi^\pm}$ equipped with the restriction of the bilinear form \eqref{eqn:bF2G} for $i=1,2$.  

Then the possible $G$-invariant, maximal totally isotropic subspaces $S \subseteq V$ are described in Tables \textup{\ref{table:cases0}}, each row representing a different $G$-equivariant isometry class.  
\end{theorem}


\begin{proof}   
First, suppose $\chi$ is self-dual.  Then up to $G$-equivariant isometry, we have $V_i \simeq \F_2(\chi)$ for $i=1,2$ with the (restriction of the) trace bilinear form \eqref{eqn:bF2G} (see also Example \ref{exm:Fq2g}).  A subspace is $G$-invariant if and only if it is an $\F_2(\chi)$-subspace; so by dimensions, a maximal isotropic subspace $S$ is generated by the $\F_2(\chi)$-span of a single vector.  We cannot have $S=V_{1}$ or $S=V_{2}$, since each $b_{i}$ is nondegenerate.  We finish with Lemma \ref{lem::isometries} and Lemma \ref{lem::isom=maxtotiso}.  Alternatively, each subspace is spanned by a unique vector $(1,\nu)$ with $\nu \in \F_2(\chi)^\times$, and one can verify directly that the $\F_2(\chi)$-span is totally isotropic if and only if $\nu\nu^*=1$, consistent with the calculation in \eqref{eqn:bxynu}.  This covers case \textsf{A}.
  
So now suppose $\chi$ is not self-dual.  Now $V_i \simeq \F_2(\chi) \oplus \F_2(\chi^*)$ for $i=1,2$, still with the trace bilinear form $b$.  Let $S \subseteq V $ be a $G$-invariant, maximal totally isotropic subspace.  Since $S$ has half of the $\F_2$-dimension of the bilinear space, as an $\F_2[G]$-module the possibilities for $S$ are $\F_2(\chi)^2$, $\F_2(\chidual)^2$, or $\F_2(\chi) \oplus \F_2(\chidual)$.  If $S \simeq \F_2(\chi)^2$, then $S =  (V_{1})_\chi \boxplus (V_{2})_{\chi} $ which is indeed totally isotropic since the restriction $b_\chi$ of $b$ to $V_{\chi}$ is identically zero by Lemma \ref{lem:OrthoDecomp1}.  Similarly for $S \simeq \F_2(\chidual)^2$; this handles cases \textsf{B} and \textsf{B}${}^\prime$.  

We now consider the possibilities for $S \simeq \F_2(\chi) \oplus \F_2(\chidual)$.  We have $S = S_\chi \oplus S_{\chi^*}$ where
\begin{equation}
\begin{aligned}
S_\chi & \colonequals \F_2(\chi) (x_1, x_2) \subseteq V_\chi, \\
S_{\chi^*} & \colonequals \F_2(\chi^*)( y_1,y_2) \subseteq V_{\chi^*},
\end{aligned}
\end{equation}
for some $(x_1,x_2),(y_1,y_2)$ where both are nonzero.  Suppose that $x_1 = 0$.  Then 
\begin{equation} 
b((0,x_2), (y_1,y_2)) = b_{1} (0,y_1) + b_{2} (x_2,y_2) = b_{2} (x_2,y_2);
\end{equation}
the nondegeneracy of $b_{2}$ on $V_{2}$ then implies that $y_2 = 0$, and $S = (V_{1})_{\chi^*} \boxplus (V_{2})_{\chi}$.  This $S$ is indeed totally isotropic since the restriction $(b_{i})_{\chi}$ of $b_i$ to $(V_{i})_\chi$ for $i = 1,2$ is identically zero by Lemma \ref{lem:OrthoDecomp1}.  Similar conclusions hold when $x_2=0$, giving cases \textsf{C} and \textsf{C}${}^\prime$.  Finally, if $x_1,x_2,y_1,y_2$ are all nonzero, then $S \cap V_1 = S \cap V_2 = \{0\}$; by Lemma  \ref{lem::isometries} and Lemma \ref{lem::isom=maxtotiso}, the number of subspaces of this form is $q-1$ and each subspace is in the same $G$-equivariant isometry class.  Alternatively, a calculation like \eqref{eqn:x1x2y1y2numu} shows that $S$ is uniquely determined by the spans of $(x_1,x_2)=(1,\nu)$ and $(y_1,y_2)=(1,\mu)$ with $\nu\mu^*=1$, giving indeed $q-1$ possibilities.

 \end{proof}

\subsection{Main result, and consequences} \label{sec:conseq}

It is now a straightforward matter to conclude our main structural result, restated here for convenience.
We recall \eqref{eqn:varphiKdef} the $2$-Selmer signature map $\varphi_K \colon \Sel_2(K) \to V(K) = V_\infty(K) \boxplus V_2(K)$; by Proposition \ref{prop:GstructV2}, we have $V_\infty(K) \simeq V_2(K) \simeq \F_2[G_K]$ as $\F_2[G_K]$-modules, equipped with the orthogonal direct sum of the bilinear forms \eqref{eqn:bF2G}.  The image 
\[ S(K) \colonequals \im(\varphi_K) \simeq \Sel_2(K)/\ker(\varphi_K) \] 
of the $2$-Selmer group under the signature map is a $G_K$-invariant maximal totally isotropic subspace of $V(K)$ by Corollary \ref{cor:Gkinv}.  By the canonical orthogonal decomposition \eqref{eqn:candecomp}, 
\begin{equation}\label{SKchipm}
S(K) = \bigboxplus_{\chi} S(K)_{\chi^{\pm}},
\end{equation} 
thus we conclude that $S(K)_{\chi^{\pm}}$ are maximal totally isotropic $G_K$-invariant subspaces of $V(K)_{\chi^{\pm}}$.

\begin{thm} \label{mainthm::Representationtheory}
Let $K$ be a Galois number field with abelian Galois group $G_K$ of odd order.  Then for each $\Ftwobar$-character $\chi$, there are exactly $6$ possibilities for $S(K)_{\chi^{\pm}} \leq V_\infty(K) \boxplus V_2(K)$ up to $G_K$-equivariant isometry.
\end{thm}

\begin{proof}
In view of the first paragraph, Theorem \ref{thm:maxisoselfdual} applies to classify the possibilities.
\end{proof}

Using the notation \eqref{eqn:notationrhow}, which matches \eqref{eqn:rhowsK} but with characters, we fill in Table \ref{table:cases1} using Table \ref{table:cases0}, Kummer duality \eqref{eqn:Kummer2}, and Proposition \ref{prop:kchi4SK}.

We now see some immediate corollaries.  First, we deduce the cornerstone results of Taylor \cite{Taylor} and Oriat \cite{Oriat2}.

\begin{corollary}[{Taylor \cite[(*)]{Taylor}, Oriat \cite[Th\'eor\`eme 2]{Oriat2}}] \label{cor:Wf2gk} 
For an abelian Galois number field $K$ of odd degree, let $\F_2(\chi)$ be an irreducible $\F_2[G_K]$-module.  Then
we have 
{\setlength{\arraycolsep}{0.3ex}
\begin{eqnarray*} 
0 \leq &k^+_\chi(K) + k^+_{\chidual}(K)& \leq  1,\\
0 \leq &k_{4,\chi}(K) + k_{4,\chidual}(K)& \leq 1.
\end{eqnarray*}}  
\flushleft Moreover, if $\F_2(\chi)$ is self-dual, then $k^+_\chi(K)=k^+_{\chidual}(K) = 0 = k_{4,\chi}(K) = k_{4,\chidual}(K)$.
\end{corollary}

\begin{proof}
Immediate from either Proposition \ref{prop:kwKW4}, using that $\rk_\chi \F_2[G]=1$ for all $\chi$, or from Table \ref{table:cases0}.
\end{proof}

\begin{remark}
Indeed, Proposition \ref{prop:kwKW4} gives the equality 
\[ k^+_\chi(K) + k^+_{\chidual}(K) = k_{4,\chi}(K) + k_{4,\chidual}(K); \]
we could not find this explicitly stated by either Taylor \cite{Taylor} or Oriat \cite{Oriat2}.
\end{remark}

Another corollary we obtain is the following result, proven by Oriat \cite{Oriat2} (and a special case of the $T$-$S$-reflection principle of Gras \cite[Th\'eor\`eme 5.18]{Gras}): see also the survey by by Lemmermeyer \cite[Theorem 7.2]{Lemmermeyer}.

\begin{corollary}[{Oriat \cite[Corollaire 2c]{Oriat2}}] \label{cor:Kabclk2}
Let $m \in \bZ_{\geq 1}$ denote the exponent of the Galois group $G_K$ for the abelian number field $K$ of odd degree. If there exists $t \in \Z$ such that $2^t \equiv -1 \pmod{m}$, then $\Cl^+(K)[2] \simeq \Cl_4(K)[2] \simeq \Cl(K)[2]$ as $\F_2[G_K]$-modules. 
\end{corollary}

\begin{proof}
By Lemma \ref{lem:self-duality}, every $\F_2[G]$-module is self-dual so the conclusion of Corollary \ref{cor:Wf2gk} implies that $k_\chi^+(K) = k_{4,\chi}(K) = 0$ for all $\chi$ in the notation of Definition \ref{def:chirank}; the result follows.  Alternatively, for the first isomorphism, apply Theorem \ref{thm:RDRP}, given that all modules are self-dual.
\end{proof}

\begin{example} If $\ell$ is an odd prime such that $2$ is a primitive root modulo $\ell$, then any cyclic number field $K$ of degree $\ell$ satisfies $\rk_2 \Cl(K) = \rk_2 \Cl^+(K)$ by Corollary \ref{cor:Kabclk2}.
\end{example}

\begin{example} 
More generally, if $2$ has even order modulo $\ell$, then Corollary \ref{cor:Kabclk2} applies to cyclic number fields of degree $\ell$. The first prime for which $2$ has even order modulo $\ell$ but $2$ is not a primitive root in $(\bZ/\ell\bZ)^\times$ is $\ell = 17$.
\end{example}

\begin{example} Corollary \ref{cor:Kabclk2} also applies to abelian groups that are not cyclic. For instance, if $K$ is a number field with Galois group $G_K \simeq \Z/3\Z \times \Z/3\Z$, then Corollary \ref{cor:Kabclk2} implies that $\rk_2 \Cl(K) = \rk_2 \Cl^+(K)$. \end{example}

\begin{remark} Edgar--Mollin--Peterson \cite[Theorem 2.5]{EMP} reprove Corollary \ref{cor:Kabclk2}, and they additionally make the claim that the corollary holds for all \emph{Galois} extensions (even though they only give a proof for the abelian case). Lemmermeyer \cite[p.~13]{Lemmermeyer} observes that this claim is erroneous.  We give an explicit counterexample (of smallest degree).  Let $K$ be the degree-27 normal closure over $\Q$ of the field $K_0$ of discriminant $3^{16} \cdot 37^{4}$ defined by
 \[ x^9 - 3x^8 - 21x^7 + 63x^6 + 141x^5 - 435x^4 - 273x^3 + 996x^2 - 192x - 64, \]
  which has LMFDB label \href{http://lmfdb.com/NumberField/9.9.80676485676081.1}{\textsf{9.9.80676485676081.1}}.
  This nonabelian extension $K$ has Galois group isomorphic to the Heisenberg group $C_3^2:C_3$ (with label 9T7), which has exponent $m = 3$.  The class group $\Cl(K)$ is trivial and $\Cl^+(K) \simeq (\Z/2\Z)^6$.
\end{remark} 

The next simplest case not treated by Corollary \ref{cor:Wf2gk} is treated by the following corollary.

\begin{corollary} \label{cor:umpmod34}
Suppose $K$ is a cyclic number field of prime degree $\ell \equiv 7\pmod{8}$ such that  $2$ has order $\frac{\ell-1}{2}$ in $(\bZ/\ell\bZ)^\times$.  Then there exist exactly \/$2$ nontrivial irreducible $\F_2[G_K]$-modules $\F_2(\chi) \not\simeq \F_2(\chidual)$.

Moreover, if $\Cl(K)[2]$ is not self-dual, then either $\Cl^+(K)[2] \simeq \F_2(\chi) \oplus \Cl(K)[2]$ or $\Cl^+(K)[2]\simeq\F_2(\chidual)\oplus \Cl(K)[2]$; and the same conclusion holds with $\Cl^+(K)[2]$ replaced by $\Cl_4(K)[2]$.
\end{corollary}

\begin{proof}
By Lemma \ref{lem:self-duality}, the hypotheses on $\ell$ imply that there is an irreducible $\F_2[G_K]$-module that is not self-dual. The first statement then follows from Example \ref{ex:prime}.
In addition, $\Cl(K)[2]$ is not self-dual if and only if $\rho_\chi(K) \neq \rho_{\chidual}(K)$, and so the second statement follows from cases \textsf{B} and \textsf{B}${}^\prime$, respectively, since $\Cl^+(K)[2] \simeq \Cl(K)[2] \oplus (S \cap V_\infty)\spcheck$ by Proposition \ref{prop:kchi4SK}.
\end{proof}

In the case that $\Cl(K)[2]$ is self-dual, there are no restrictions on $\Cl^+(K)[2]$, and we model this situation in Conjecture \ref{conj:3mod4}.

\smallskip

\subsection{Unit signature ranks}

We now deduce some consequences for unit signature ranks.  Recall that the unit signature rank of $K$ is $\sgnrk(\calO_K^\times) = \dim_{\F_2} \sgn_\infty(\calO_K^\times),$
where $\sgn_\infty$ was defined in Definition \ref{def:VooK}. 
There is a natural exact sequence
\begin{equation*}
    1 \to \{\pm 1\}^n/\sgn_\infty(\calO_K^\times) \to \Cl^+(K) \to \Cl(K) \to 1 
\end{equation*} 
tying together the unit signature rank and the isotropy rank. For example, we have
\begin{equation*} \label{eq:tworankandsgnrk}
n-\sgnrk(\calO_K^\times) \leq \rk_2 \Cl^+(K).
\end{equation*} 
In addition, recall from \eqref{eq:sgnrkchi} that $\sgnrk_\chi(\calO_K^\times)$ is equal to the multiplicity of $\F_2(\chi)$ in $\sgn_\infty(\calO_K^\times)$. 
Since $\calO_{K}^\times/(\calO_K^{\times})^2 \simeq \F_2[G_K]$, and $G_K$ is abelian, every irreducible $\F_2[G_K]$-module occurs with multiplicity $1$ inside the unit group and hence 
\begin{equation}\label{eq:sgnrkineq}
 0 \leq \sgnrk_\chi(\calO_K^\times) \leq 1. 
\end{equation}
We improve upon the above inequality in the following main result.

\begin{theorem} \label{thm:oddm10}
Let $K$ be an abelian number field of odd degree with Galois group $G_K$.  Let $\chi$ be an $\Ftwobar$-character of $G_K$.  Then the following statements hold.
\begin{enumalph}
\item If $k_\chi^+(K) = 1$, then $$\sgnrk_\chi(\calO_K^\times) = 0.$$
\item If $k_\chi^+(K) = 0$, then $$\max\bigl(0, 1- \rho_\chi(K)\bigr) \leq \sgnrk_\chi(\calO_K^\times) \leq 1.$$
\end {enumalph}
\end{theorem}

\begin{proof}
The statement $\sgnrk_\chi(\calO_K^\times) = 0$ is equivalent to $\left ( \calO_{K}^\times/(\calO_{K}^{\times})^2 \right ) _\chi \subseteq \ker(\varphi_{K,\infty})_\chi$. We can determine $\ker(\varphi_{K,\infty})_\chi$ by combining Theorem \ref{thm:RDRP} with Lemma \ref{lem::AllKummerPairings} to get
\[ \ker(\varphi_{K,\infty}) \simeq \Cl_4(K)[2]\spcheck \simeq \Cl^+(K)[2]. \] 
Hence, the $\F_2(\chi)$-multiplicities of $\ker(\varphi_{K,\infty})_\chi \subseteq \Sel_2(K)_\chi$ are given as follows:
 \begin{equation*}
\begin{aligned}
\rk_\chi \Sel_2(K) & = \rho_\chi(K) +1\\
\rk_\chi \ker(\varphi_{K,\infty}) & =  \rho_\chi(K) + k_{\chi}^+(K) =\rho_\chi^+(K).
\end{aligned} 
\end{equation*}
When $k_{\chi}^+(K) = 1$, then $(\calO_{K}^\times/\calO_{K}^{\times 2} )_\chi  \subseteq \Sel_2(K)_\chi = \ker(\varphi_{K,\infty})_\chi$ and so $\sgnrk_\chi(\calO_K^\times) = 0$. This establishes (a). For (b), observe that in order for $(\calO_{K}^\times/\calO_{K}^{\times 2} )_\chi \subseteq \ker(\varphi_{K,\infty})_\chi$, we would need to have $\rk_\chi \ker(\varphi_{K,\infty})_\chi \neq 0$ which does not occur if $k_{\chi}^+(K) = \rho_\chi(K) = 0$. 
\end{proof}

\begin{example}
Let $\chi$ be the trivial character so that $\F_2(\chi) \simeq \F_2$ is the trivial $\F_2[G_K]$-module. Then $\sgnrk_\chi(\calO_K^\times) = 1$: indeed, $-1$ generates the unique subspace of $\calO_K^\times/(\calO_K^{\times})^2$ with trivial action.  To see that this accords with Theorem \ref{thm:oddm10}, note that from Lemma \ref{lem:cl2} we have $\rho^+_\chi(K) = \rho_\chi(K) = 0$ so $k_\chi^+ = 0$ and hence $\sgnrk_\chi(\calO_K^\times) = 1$.\end{example}

Summing the contributions of each irreducible gives the following corollary.  

\begin{corollary} \label{cor:Kcycsgrkyup}
Let $K$ be a cyclic number field of odd prime degree $\ell$, and let $f$ be the order of $2$ modulo $\ell$.  Then
\[ \sgnrk(\calO_K^\times) \equiv 1 \pmod{f} \]
and the following statements hold:
\begin{enumalph}
\item If $f$ is odd, then 
\[ \max\left(1,{\textstyle\frac{\ell +1}{2}} - \rk_2\Cl(K)\right)  \leq \sgnrk(\calO_K^\times) \leq \ell. \]
\item If $f$ is even, then 
\[ \max\bigl(1,\ell - \rk_2 \Cl(K)\bigr)  \leq \sgnrk(\calO_K^\times) \leq \ell. \]
\end{enumalph}
\end{corollary}

\begin{proof}
By Example \ref{ex:prime}, all nontrivial irreducible $\F_2[G_K]$-modules have cardinality $2^f$, and so together with the trivial component generated by $-1$ gives the first congruence.  

The upper bounds in (a) and (b) are immediate since $\rk_2 V_\infty = \ell.$  To prove (b), note that all $\F_2[G_K]$-modules are self-dual by Lemma \ref{lem:self-duality}, hence Corollary \ref{cor:Wf2gk} implies that $k_{\chi}^+(K) = 0$. By adding up Theorem \ref{thm:oddm10}(b) for all $1 + \frac{\ell-1}{f}$ irreducible $\F_2[G_K]$-modules as in Example \ref{ex:prime}, we conclude the result. 

We conclude by proving statement (a) by considering $\chi$ and $\chidual$ together.  Every nontrivial $\F_2[G]$-module is non-self-dual by Lemma \ref{lem:self-duality}.  We refer to the cases in Table \ref{table:cases0}.  We claim that for every nontrivial character $\chi$, we have 
\begin{equation} \label{eqn:yupsgnrkbnd}
1-\rho_\chi(K)-\rho_{\chidual}(K) \leq \sgnrk_\chi(\calO_K^\times) + \sgnrk_{\chidual}(\calO_K^\times).  
\end{equation}
Indeed, if $k_\chi^+(K)=1$ (cases \textsf{B}${}^\prime$/\textsf{C}${}^\prime$), then $k^+_{\chidual}(K)=0$ by Corollary \ref{cor:Wf2gk}, and $\sgnrk_\chi(\calO_K^\times) = 0$ by Theorem \ref{thm:oddm10}(a) so by Theorem \ref{thm:oddm10}(b) we have 
\[ 1-\rho_\chi(K)-\rho_{\chidual}(K) \leq 1-\rho_{\chidual}(K) \leq \sgnrk_{\chidual}(\calO_K^\times) = \sgnrk_{\chi}(\calO_K^\times) + \sgnrk_{\chidual}(\calO_K^\times).  \]
By symmetry, the same conclusion holds when $k_{\chidual}^+(K)=1$ (cases \textsf{B}/\textsf{C}).  In the remaining case \textsf{D} where $k_\chi^+(K)=k_{\chidual}^+(K)=0$, summing Theorem \ref{thm:oddm10}(b) twice gives
\[ 1-\rho_\chi(K)-\rho_{\chidual}(K) < 2-\rho_\chi(K)-\rho_{\chidual}(K) \leq \sgnrk_\chi(\calO_K^\times) + \sgnrk_{\chidual}(\calO_K^\times). \]
This proves the claim in all cases.  

Summing \eqref{eqn:yupsgnrkbnd} over the $(\ell-1)/(2f)$ pairs of irreducible nontrivial $\F_2[G_K]$-modules as well as the trivial $\F_2[G_K]$-module, then gives
\[ \frac{\ell+1}{2}-\rk_2\Cl(K) \leq \sgnrk(\calO_K^\times). \qedhere \]
\end{proof}

We record the following special case of Corollary \ref{cor:Kcycsgrkyup} (observed for $\ell=3$ by Armitage--Fr\"ohlich \cite[Theorem V]{ArmitageFrohlich}).

\begin{corollary} \label{cor:1orell}
If $K$ is a cyclic number field of prime degree $\ell$ where $2$ is a primitive root modulo $\ell$, then $\sgnrk(\calO_{K}^\times) = 1$ or $\ell$. If the class number of $K$ is odd, then $\sgnrk(\calO_{K}^\times) = \ell$.
\end{corollary}

The above setup allows us to recover many other related statements.  We illustrate with the following.  

\begin{theorem}[{Ichimura \cite[Theorem 2]{Ichimura}}] Let $K$ be an abelian number field of odd degree with Galois group $G_K$.  Let $\chi$ be an $\Ftwobar$-character of $G_K$.  Then the following statements are equivalent:
\begin{enumerate}
\item[\textup{(i)}] $\left ( \calO_{K}^\times/(\calO_{K}^{\times})^2 \right ) _{\chi^{\pm}}  \cap \ker(\sgn_\infty) \neq \{0\}$.
\item[\textup{(ii)}] $\left ( \calO_{K}^\times/(\calO_{K}^{\times})^2 \right ) _{\chi^{\pm}}  \cap \ker(\sgn_2) \neq \{0\}$.
\end{enumerate}
\end{theorem}
\begin{proof}
This statement is trivially true whenever $\left ( \calO_{K}^\times/(\calO_{K}^{\times})^2 \right ) _{\chi^{\pm}}  \cap \ker(\varphi_{K}) \neq \{0\}$.  Otherwise, we have  $S(K)_{\chi^{\pm}} = \varphi_K \bigl( \left( \calO_{K}^\times/(\calO_{K}^{\times})^2 \right ) _{\chi^{\pm}} \bigr)$, so the statement is $S(K)_{\chi^\pm} \cap V_\infty(K) = \{0\}$ if and only if $S(K)_{\chi^\pm} \cap V_2(K) = \{0\}$ and then it is equivalent to Proposition \ref{prop:kwKW4}. 
\end{proof}

\section{Conjectures} \label{sec:models}

Even with many aspects determined in a rigid way by the results of the previous section, there still remain scenarios where randomness remains.  In this section, we propose a model in the spirit of the Cohen--Lenstra heuristics for this remaining behavior.

\subsection{Isotropy ranks}

We begin by developing a model for isotropy ranks when $K$ runs over a collection of $G$-number fields (i.e., Galois number fields $K$ equipped with an isomorphism such that $\Gal(K\,|\,\bQ) \simeq G$), where $G$ is a fixed finite abelian group of odd order. In light of Theorem \ref{mainthm::Representationtheory} (and Table \ref{table:cases0}) a heuristic is only necessary to distinguish cases \textsf{C},\textsf{C}${}^\prime$ from \textsf{D}, i.e., when $\chi$ is a non-self-dual $\Ftwobar$-character of $G$ and the collection is restricted to those $K$ such that $\rho_\chi(K) = \rho_{\chidual}(K)$.  For all other cases, the isotropy ranks are determined. 

We make the following heuristic assumption:
\begin{itemize}
\item[(H1)]\label{item:H1} For the collection of $G$-number fields $K$ such that $\rho_\chi(K)=\rho_{\chidual}(K)$, the image component $S(K)_{\chi^{\pm}}$ as defined in \eqref{SKchipm} is distributed as a uniformly random $G$-invariant maximal totally isotropic subspace of $\F_2[G]^2_{\chi^{\pm}}$ (see Example \ref{ex:mult1}).
\end{itemize}

The assumption (\hyperref[item:H1]{H1}), combined with the restrictions and masses in Table \ref{table:cases0} lead us to one of our main conjectures. 

\begin{conjecture}\label{mainconj::Tworanks0} 
Let $G$ be an odd finite abelian group, and let $\chi$ be a non-self-dual $\Ftwobar$-character of $G$ with underlying module of cardinality $\#\F_2(\chi) = q$.  Then as $K$ varies over $G$-number fields such that $\rho_\chi(K) = \rho_{\chidual}(K)$, we have:
{\setlength{\arraycolsep}{0.3ex}
\begin{eqnarray*}
\Prob \bigl(k^+_\chi(K)+k^+_{\chidual}(K) = 0 \bigr)   &=& \frac{q-1 }{ q +1}; \\ 
\Prob \bigl (k^+_\chi(K)+k^+_{\chidual}(K) = 1 \bigr ) &=&  \frac{2 }{q +1}. 
\end{eqnarray*}}
\end{conjecture}

The same heuristic implies the same conjecture for the $2$-adic isotropy ranks; indeed by Proposition \ref{prop:kwKW4}, we have $k^+_\chi(K)+k^+_{\chidual}(K) = k_{4,\chi}(K)+k_{4,\chidual}(K)$.  A particularly simple case of Conjecture \ref{mainconj::Tworanks0} is complementary to Corollary \ref{cor:umpmod34}.

\begin{conjecture}\label{conj:3mod4}
Let $G = \Z/\ell\Z$ where $\ell \equiv 7 \pmod{8}$ is prime and suppose $2$ has order $\frac{\ell-1}{2}$ in $(\bZ/\ell\bZ)^\times$, and let $q \colonequals 2^{\frac{\ell-1}{2}}$.  As $K$ varies over $G$-number fields such that $\Cl(K)[2]$ is a self-dual $\F_2[G_K]$-module, we have
\begin{eqnarray*}
\Cl^+(K)[2]  \simeq 
\begin{cases}  \Cl(K)[2] & \text{with probability} \quad  \frac{q-1}{q+1}; \\
 \F_2(\chi) \oplus \Cl(K)[2] & \text{with probability} \quad   \frac{1}{q+1}; \\
  \F_2(\chidual) \oplus \Cl(K)[2] & \text{with probability} \quad  \frac{1}{q+1}. \end{cases}
\end{eqnarray*}
where $\chi$ is a nontrivial $\Ftwobar$-character of $G_K$.
\end{conjecture}

We predict the same probabilities as in Conjectures \ref{mainconj::Tworanks0} and \ref{conj:3mod4} in other natural subfamilies, such as when we fix the value $\rho_\chi(K)=\rho_{\chidual}(K)=\vrho$; in particular,  this includes the family of $G$-number fields with odd class number (i.e., those with $\rk_2 \Cl(K)=0$). 

\subsection{Unit signature ranks}

We now extend these heuristics to the unit signature rank.  Recall that the $2$-Selmer group $\Sel_2(K)$ is an $\F_2[G_K]$-module containing $\cO^\times_K/(\cO_K^{\times})^2$ and the subspace $\ker(\varphi_{K,\infty}) \subseteq \Sel_2(K)$ of totally positive elements.   To study the distribution of the units, for each odd order abelian group $G$, we make the following heuristic assumption:
\begin{itemize}
\item[(H2)\label{item:H2}] For the collection of $G$-number fields $K$, the subspace of $\Sel_2(K)$ generated by $\calO_K^\times/(\calO_K^{\times})^2$ is distributed as a uniformly random $\F_2[G]$-submodule isomorphic to $\F_2[G]$ containing $-1$.  
\end{itemize} 

Decomposing into irreducibles, since $\cO_K^\times/(\cO_K^{\times})^2 \simeq \F_2[G]$ we have $\rk_\chi \cO_K^\times =1$ for each irreducible $\F_2(\chi)$, and so we might also make the heuristic assumption:
\begin{itemize}
\item[(H2${}^\prime$)]\label{item:H2prime}
For the collection of $G$-number fields $K$ and for each \emph{nontrivial} $\Ftwobar$-character of $G$, the subspace of $\Sel_2(K)_\chi$ generated by $(\calO_K^\times/(\calO_K^{\times})^2)_\chi$ is distributed as a uniformly random, $1$-dimensional $\F_2(\chi)$-subspace.  
\end{itemize}
Note that (\hyperref[item:H2]{H2}) is equivalent to (\hyperref[item:H2prime]{H2${}^\prime$}) and an \emph{independence  assumption} for each $\F_2(\chi)$, i.e., we expect no extra structure relating different isotypic components of the units inside $\Sel_2(K)$.

\begin{remark}
To make assumption (\hyperref[item:H2]{H2}), we consider $\cO_K^\times/(\cO_K^{\times})^2 \subseteq \Sel_2(K)$ and we do \emph{not} look at the 2-Selmer map $\varphi_{K,\infty}$. The pairing and duality relations that put restrictions on the $k_\chi(K)$'s as in Corollary \ref{cor:Wf2gk} will have an effect on the subspace $\ker(\varphi_{K,\infty}) \subseteq \Sel_2(K)$; in particular, it will impose constraints on the isotypic components of $\ker(\varphi_{K,\infty})$. However, we model the subspace $\ker(\varphi_{K,\infty})$ independently of $\cO_K^\times/(\cO_K^{\times})^2$, so there are no restrictions on $(\cO_K^\times/(\cO_K^{\times})^2)_\chi$ inside $\Sel_2(K)_\chi$.
\end{remark}

We now state a conjecture for collections of $G$-number fields that are not completely determined by Theorem \ref{thm:oddm10}. 
We recall from \eqref{eq:sgnrkineq} that $\sgnrk_\chi(\calO_K^\times)=0$ or $1$ for any $\Ftwobar$-character $\chi$ of $G_K$.

\begin{conjecture} \label{conj::ConjUnitSignatureRanks}
Let $G$ be an abelian group of odd order and let $\chi$ be an $\Ftwobar$-character of $G$ with $q \colonequals \#\F_2(\chi)$.  As $K$ varies over $G$-number fields such that $\rk_\chi \Cl^+(K) = \rk_\chi \Cl(K) = \vrho$, we have
     \[ \Prob\bigl(\sgnrk_\chi(\calO_K^\times) = 0\bigr) = \frac{q^{\vrho}-1}{q^{\vrho+1}-1}; \]
     \[ \Prob\bigl(\sgnrk_\chi(\calO_K^\times) = 1\bigr) = \frac{q^{\vrho+1}-q^{\vrho}}{q^{\vrho+1}-1}. \]
\end{conjecture}

\begin{proof}[Proof assuming \textup{(\hyperref[item:H2prime]{H2${}^\prime$})}.]
The dimensions of the isotypic components are given as follows:
\begin{itemize}
\item $\rk_{\chi} \calO_K^{\times}/(\calO_{K}^\times)^{2} =1$;
\item $\rk_\chi  \ker(\varphi_{K,\infty}) = r$; and
\item $\rk_\chi  \Sel_2(K) = r+1$.
\end{itemize}
 Therefore, under (\hyperref[item:H2prime]{H2${}^\prime$}) we would have
 \begin{equation*}
     \begin{aligned}
     \Prob\bigl((\calO_K^{\times}/(\calO_{K}^\times)^{2})_\chi \subseteq \ker(\varphi_{K,\infty})_\chi\bigr) &= \frac{\text{\#\{1-dimensional subspaces of } \ker(\varphi_{K,\infty})_\chi \} }{\#\{\text{1-dimensional subspaces of } \Sel_2(K)_\chi \}} \\
     &= \frac{(q^{\vrho}-1)/(q-1)}{(q^{\vrho+1}-1)/(q-1)} = \frac{q^{\vrho}-1}{q^{\vrho+1}-1}
     \end{aligned}
     \end{equation*}
     as claimed.
     \end{proof}
     
We now turn to the simplest case, where $G$ is cyclic of prime order $\ell$ and $2$ is a primitive root mod $\ell$.  By Corollary \ref{cor:1orell}, we conclude that $\sgnrk(\calO_K^\times)=1$ or $\ell$ and $\rk_2 \Cl(K) \equiv 0 \pmod{\ell - 1}$ by Lemma \ref{lem:cl2}

\begin{conjecture} \label{conj:sgnrkcondi} 
Let $\ell$ be an odd prime such that $2$ is a primitive root modulo $\ell$, and let $q \colonequals 2^{\ell-1}$. If $r \in \Z_{\geq 0}$,  then as $K$ ranges over cyclic number fields of degree $\ell$ with $\rk_2 \Cl(K) = (\ell-1)r$, we have
\begin{equation*} 
\Prob\bigl(  \sgnrk(\cO_K^\times) = 1  \bigr)  = \frac{q^r-1}{q^{r+1}-1}; \qquad
\Prob\bigl(  \sgnrk(\cO_K^\times) = \ell  \bigr)  = \frac{q^{r+1} - q^{r}}{q^{r+1}-1}.
\end{equation*}
\end{conjecture}

\begin{proof}[Proof assuming \textup{(\hyperref[item:H2prime]{H2${}^\prime$})}.]
There is a unique nontrivial $\Ftwobar$-character $\chi$ of $\bZ/\ell\bZ$, and so $\rk_2 \Cl(K) = (\ell - 1)r$ if and only if $\rk_\chi \Cl(K) = r$.  The result follows then from the hypothesis ({H2${}^\prime$}).
\end{proof}

\noindent Conjecture \ref{conj:sgnrkcondi} is a theorem for the case $r=0$ (odd class number) by Corollary \ref{cor:1orell}.

In a different direction, we can consider the class of fields where not all modules are self-dual. 
The most common case is expected to be fields with odd class number which by Theorem \ref{thm:oddm10} has $\sgnrk_\chi(\calO_K^\times) = 1 - k_\chi^+(K)$. 
Using Conjecture \ref{mainconj::Tworanks0}  and summing over the contributions we end up with the following binomial distribution.

\begin{conjecture} \label{conj:ellm1fq}
Let $\ell$ be an odd prime, let $f$ be the order of $2$ modulo $\ell$, and suppose that $f$ is odd.  Let $q\colonequals 2^f$ and $m \colonequals {\frac{\ell-1}{2f}} \in \Z_{>0}$.  
Then as $K$ varies over cyclic number fields of degree $\ell$ with \emph{odd} class number, we have
\[ \Prob\Bigl(\sgnrk(\calO_K^\times)=fs+\frac{\ell+1}{2}\Bigr) =  \binom{m}{s} \left ( \frac{q -1 }{q +1} \right )^s \left (   \frac{2 }{q +1} \right )^{m-s} \]
for $0 \leq s \leq m$.
\end{conjecture}

\subsection{Applications of class group heuristics for cyclic cubic and quintic fields}

The conjectures in the previous section give predictions conditioned on the $2$-rank of the class group.  We next combine our conjectures with predictions for the latter by applying the conjectures of \cite{Malle} correcting the Cohen--Lenstra heuristics for cyclic cubic and quintic fields.   

\smallskip

For $m \in \Z_{\geq 0} \cup \{\infty\}$ and $q \in \R_{>1}$, define $(q)_0 \colonequals 1$ and for nonzero $m$, let
\begin{equation*} 
(q)_m \colonequals \prod_{i=1}^m (1-q^{-i}).
\end{equation*}
For cyclic fields of odd prime degree $\ell$, Cohen--Lenstra \cite{CohenLenstra} made a prediction for the $2$-part of their class groups; in particular, they imply (in the first moment) that the average size of $\Cl(K)[2]$ is equal to $\left(1+2^{-f}\right)^{\frac{\ell-1}{f}},$ where $f$ is the order of $2$ in $(\bZ/\ell\bZ)^\times$.  However, computations by Malle \cite{Malle} suggest that this prediction needs a correction for the fact that the second roots of unity (but not the fourth roots of unity) are contained in any such field.  Malle \cite[(1),(2)]{Malle} goes on to make predictions for the distribution of $\rk_2 \Cl(K)$ as $K$ ranges over cyclic fields of degrees $3$ and $5$.

\begin{conjecture}[Malle] \label{conj:quoteMallequote}
Let $\ell = 3$ or $5$, and let $q=2^{\ell-1}$.  Then as $K$ ranges over cyclic number fields of degree $\ell$, we have
\begin{equation} \label{eqn:problmalle}
\Prob\bigl(\rk_2 \Cl(K) = (\ell-1)r\bigr) = \left(1+\frac{1}{\sqrt{q}}\right) \frac{(\sqrt{q})_\infty (q^2)_\infty}{(q)_\infty^2} \cdot \frac{1}{\sqrt{q}^{r(r+2)}\cdot(q)_{r}}
\end{equation}
for all $r \in \Z_{\geq 0}$. 
\end{conjecture}

Note that under the hypotheses of Conjecture \ref{conj:quoteMallequote}, we have $\rk_2 \Cl(K) = \rk_2 \Cl^+(K)$ by Corollary \ref{cor:Kabclk2}, hence the left-hand side of \eqref{eqn:problmalle} is equal to $\Prob\bigl(\rk_2 \Cl^+(K) = (\ell-1)r\bigr)$. (For a discussion about class group heuristics for cyclic fields of prime degree $\ell \geq 7$, see Remark \ref{rem:malle}.) 

Combining Conjecture \ref{conj:quoteMallequote} with Conjecture \ref{conj:sgnrkcondi} and summing gives the following:

\begin{conjecture} \label{conj:ell35sgkr0}
 As $K$ varies over cyclic number fields of degree $\ell = 3$ or $5$, we predict
\begin{equation*}
    \begin{aligned}
\Prob\bigl(  \sgnrk(\calO_K^\times) = 1  \bigr)  &= \left(1+\frac{1}{\sqrt{q}}\right) \cdot \frac{(\sqrt{q})_\infty (q^2)_\infty}{(q)_\infty^2} \cdot \sum_{r = 0}^\infty \frac{1}{\sqrt{q}^{r(r+2)} \cdot(q)_r} \cdot \frac{q^{r}-1}{q^{r+1}-1},
\end{aligned}
\end{equation*}
where $q = 2^{\ell-1}$. \end{conjecture}
Approximate numerical values of these probabilities are given in the following table: 
\begin{center}
\begin{tabular}{c|cc}
& $\ell = 3$ & $\ell = 5$  \\ 
\hline\vspace{-10pt}\\
$\sgnrk(\calO_K^\times) = 1$ & $0.029573$  &  $0.000965$  \\\vspace{-10pt}\\ 
$\sgnrk(\calO_K^\times) = \ell$ & $0.970427$ & $0.999035$ \\
\end{tabular}
\end{center}
\subsection{Summary of results in small degree} \label{sec:subsumm}

We now summarize the results and conjectures for the case $\ell=3$, $5$, and $7$.

\subsubsection*{Cyclic cubic fields} 

We begin with the case $G=\Z/3\Z$ and $\ell=3$.  Here, $2$ is a primitive root, and so there is a unique nontrivial irreducible $\F_2[G]$-module with $\F_2$-dimension $\ell-1 = 2$ implying that $\rk_2 \Cl(K)$ is always even. 
Malle \cite[(1)]{Malle} (as in Conjecture \ref{conj:quoteMallequote}) predicts
\[ \Prob\bigl(\rk_2 \Cl(K) = 0,2,4\bigr) \approx 85.30\%, 14.21\%, 0.47\%; \]
the remaining cyclic cubic fields (having $\rk_2 \Cl(K) \geq 6$) conjecturally comprise less than $0.004\%$ of all cyclic cubic fields.  By Corollary \ref{cor:Kabclk2}, we have $\Cl(K)[2] \simeq \Cl^+(K)[2]$.  

In this case, Conjecture \ref{conj:sgnrkcondi} predicts $\Prob\bigl(\sgnrk(\calO_K^\times) = s \,|\, \rk_2 \Cl(K) =\vrho\bigr)$ according to the following table: 
\vspace{-5pt}
\begin{center}
\begin{tabular}{c|ccc}
& $\vrho =0$ & $\vrho=2$ & $\vrho=4$  \\ 
\hline\vspace{-10pt}\\
$s=1$ & $0$ & $\frac{1}{5}$ & $\frac{5}{21}$ \\\vspace{-10pt}\\ 
$s=3$ & $1$ & $\frac{4}{5}$ & $\frac{16}{21}$\vspace{5pt}\\
\end{tabular}
\end{center}
For example, amongst cyclic cubic fields with $\rk_2 \Cl(K)=4$, we predict $\frac{16}{21}$ will have units of mixed signature.  Combining these first three values for the $2$-ranks with the associated conditional probabilities for $\sgnrk(\calO_K^\times)$ yields
$$\Prob\bigl(  \sgnrk(\calO_K^\times) = 3  \bigr) \approx   1\cdot85.30\% + \frac{4}{5}\cdot14.21\% +  \frac{16}{21}\cdot0.47\%  \approx 97.03\%.  $$
Conjecture \ref{conj:ell35sgkr0} then implies: as $K$ varies over cyclic cubic fields, the unit signature rank is equal to $1$ approximately $3\%$ of the time, and the unit signature rank is equal to $3$ approximately $97\%$ of the time.

\subsubsection*{Cyclic quintic fields} 

When $G = \bZ/5\bZ$, we again have that $2$ is a primitive root modulo $5$, so there is a unique irreducible nontrivial irreducible $\F_2[G_K]$-module of dimension $4$.   Malle \cite[(2)]{Malle} predicts 
\[ \Prob\bigl(\rk_2\Cl(K)=0,4,8\bigr) \approx 98.359\%, 1.639\%, 0.002\% \]
and $\Prob\bigl(\rk_2\Cl(K) \geq 8\bigr) \leq 0.02\%$.  Again, by Corollary \ref{cor:Kabclk2} we have $\Cl(K)[2] \simeq \Cl^+(K)[2]$.  Here, Conjecture \ref{conj:sgnrkcondi} predicts $\Prob\bigl(\sgnrk(\calO_K^\times) = s \,|\, \rk_2\Cl(K)=\vrho\bigr)$ as:
\begin{center}
\begin{tabular}{c|ccc}
 & $\vrho=0$ & $\vrho=4$ & $\vrho=8$  \\
\hline \vspace{-10pt}
\\ $s=1$ & $0$ & $\frac{1}{17}$ & $\frac{17}{273}$ \\\vspace{-10pt}\\
$s=5$ & $1$ & $\frac{16}{17} $ &  $\frac{256}{273}$ \\
\end{tabular}
\end{center}
Summing as above yields
$$ \Prob\bigl(  \sgnrk(\calO_K^\times) = 5  \bigr) \approx   1\cdot98.35\% + \frac{16}{17} \cdot 1.63\% +  \frac{256}{273}\cdot0.002\% \approx 99.90\%,  $$
and so Conjecture \ref{conj:ell35sgkr0} predicts that $99.9\%$ of cyclic quintic fields have units of all possible signatures, and indeed they are abundant.  For $\sgnrk(\calO_K^\times) = 1$, we find the cyclic quintic field $K = \Q(\alpha)$ of conductor $39821$ and discriminant $39821^4$, where $\alpha$ is a root of the polynomial 
$$ x^5 + x^4 - 15928x^3 - 218219x^2 + 20800579x + 363483463. $$

\subsubsection*{Cyclic septic fields}

We now consider the case $G=\Z/7\Z$.  Since $2$ has order $3$ modulo $7$ and $-1 \not \in \langle 2 \rangle \leq (\Z/7\Z)^\times$, there are precisely two nontrivial irreducibles $\F_2(\chi) \not\simeq \F_2(\chidual)$ and $\#\F_2(\chi)=\#\F_2(\chidual)=2^{3}=8$.  We refer to the cases in Table \ref{table:cases0}; case \textsf{A} does not occur because $\chi$ is not self-dual.  We have $\rk_2 \Cl(K)=\rk_2 \Cl(K)_{\chi^{\pm}} = 3(\rho_\chi(K)+\rho_{\chidual}(K))$.
\begin{itemize}
\item Cases \textsf{B}, \textsf{B}${}^\prime$ are exactly those where $\rho_\chi \neq \rho_{\chidual}$, i.e., $\Cl(K)[2]$ is not self-dual, in which case $\rho_\chi-\rho_{\chidual}=\pm 1$.  In these cases, $k^+(K)=3(k_\chi^+(K)+k_{\chidual}^+(K))=3$, i.e., $\Cl^+(K)[2] = \Cl(K)[2] \oplus (\Z/2\Z)^3$.
\item The remaining cases \textsf{C}, \textsf{C}${}^\prime$, and \textsf{D} are those where $\Cl(K)[2]$ is self-dual.  For such fields, we have $k^+(K)=0$ (in case \textsf{D}) or $k^+(K)=3$ (in cases \textsf{C} or \textsf{C}${}^\prime$), and Conjecture \ref{mainconj::Tworanks0} predicts that 
\begin{eqnarray*} \Prob\bigl(k^+(K)= 3\bigr) & = & \textstyle{\frac{2}{9}}, \\
\Prob\bigl(k^+(K)= 0\bigr) & = & \textstyle{\frac{7}{9}}.
\end{eqnarray*} 
\end{itemize}
In particular, $\Cl(K)[2]$ is self-dual if and only if $\rk_2 \Cl(K)$ is even.

\begin{example} \label{exm::cases7}
We now provide examples of the above three cases for cyclic septic number fields. For each case let $K = \Q(\alpha)$ where $\alpha$ is a root of the polynomial $f(x)$.
\begin{itemize}
    \item For the field with LMFDB label \href{http://www.lmfdb.org/NumberField/7.7.14011639427134441.1}{\textsf{7.7.14011639427134441.1}} of discriminant $491^6$ defined by $f(x)=x^7 - x^6 - 210x^5 - 1423x^4 - 1410x^3 + 8538x^2 + 9203x - 19427$, we have $\Cl(K)[2] = (\Z/2\Z)^3$ and $\Cl^+(K)[2] = (\Z/2\Z)^6$ so that $\Cl(K)[2]$ is not self dual, is an example of case \textsf{B}/\textsf{B}${}^\prime$, and $k^+(K) = 3$. 
     \item For \href{https://www.lmfdb.org/NumberField/7.7.594823321.1}{\textsf{7.7.594823321.1}} of discriminant $29^6$ defined by $f(x) = x^7 - x^6 - 12x^5 + 7x^4 + 28x^3 - 14x^2 - 9x - 1$, we have $\Cl(K)[2] = 1$ and $\Cl^+(K)[2] = (\Z/2\Z)^3$ so that $\Cl(K)[2]$ is self dual, is an example of case \textsf{C}/\textsf{C}${}^\prime$, and $k^+(K) = 3$.
     \item For \href{http://www.lmfdb.org/NumberField/7.7.6321363049.1}{\textsf{7.7.6321363049.1}} of discriminant $43^6$ defined by $f(x)=x^7 - x^6 - 18x^5 + 35x^4 + 38x^3 - 104x^2 + 7x + 49$, we have $\Cl(K)[2] = \Cl^+(K)[2] = 1$ so that $\Cl(K)[2]$ is self dual, is an example of case \textsf{D}, and $k^+(K) = 0$.
     \end{itemize}
\end{example}

For unit signature ranks, using the formulas in Conjectures \ref{conj::ConjUnitSignatureRanks} and \ref{conj:ellm1fq} we make the following predictions for class groups of cyclic septic fields with low 2-rank.  

\begin{itemize}
\item Suppose $\rk_2\Cl(K)=0$. Conjecture \ref{conj:ellm1fq} then implies:
\begin{eqnarray*} 
\Prob\bigl(  \sgnrk(\calO_K^\times) = 4 \,|\, \rk_2 \Cl(K) =0 \bigr)  &=& \textstyle{\frac{2}{9}};\\\Prob\bigl(  \sgnrk(\calO_K^\times) = 7 \,|\, \rk_2 \Cl(K)=0 \bigr)  &=& \textstyle{\frac{7}{9}}.
\end{eqnarray*}
\item Suppose $\rk_2\Cl(K)=3$.  Without loss of generality, assume $\rho_{\chi}(K) = 1$ and $\rho_{\chidual}(K) =0$.  By Theorems  \ref{mainthm::Representationtheory}(b)(i) and \ref{thm:oddm10}(a), we have $\sgnrk_{\chidual}(\calO_K^\times) = 0$. Using Conjecture \ref{conj::ConjUnitSignatureRanks} with $\rho_{\chi}(K) =1$, we predict that $\sgnrk_{\chi}(\calO_K^\times) = 0$ occurs with probability $\frac{7}{63}$, so
\begin{eqnarray*}
\Prob\bigl(  \sgnrk(\calO_K^\times) = 1  \,|\, \rk_2\Cl(K)=3 \bigr)  &=& \textstyle{\frac{1}{9}};\\
\Prob\bigl(  \sgnrk(\calO_K^\times) = 4  \,|\, \rk_2\Cl(K)=3 \bigr)  &=& \textstyle{\frac{8}{9}}.
\end{eqnarray*}
\end{itemize}

\section{Computations}  \label{sec:calcs}
In this section, we present computations that provide evidence to support our conjectures.  To avoid redundancy, instead of working with families of $G$-number fields (which weights each isomorphism class of a field $K$ by $\#\Aut(G_K)$), we weight each isomorphism class of number fields by $1$.  (Either weighting evidently gives the same probabilities and moments.)

We begin by describing a method for computing a random cyclic number field of odd prime degree $\ell$ of conductor $\leq X$.  Recall (by the Kronecker--Weber theorem) that $f \in \Z_{\geq 0}$ arises as a conductor for such a field if and only if $f=f'$ or $\ell^2 f'$ where $f'$ is a squarefree product of primes $p \equiv 1 \pmod{\ell}$.  Moreover, the number of such fields is equal to $(\ell-1)^{\omega(f)-2}$ if $\omega(f) \geq 2$ and $\ell \mid f$, otherwise the number is $(\ell-1)^{\omega(f)-1}$.  Our algorithm generates a random factored integer $f \leq X$ of this form and a uniform random character with given conductor; then, it constructs the corresponding field by computing an associated Gaussian period.

\subsection{Cubic fields} \label{sec:cubfields}

We sampled cyclic cubic fields in this manner, performing our computations in \Magma\ \cite{Magma}; the total computing time was a few CPU days.  The class group and narrow class group computations are conjectural on the Generalized Riemann Hypothesis (GRH).  Our code generating this data is available online \cite{ourlist}. 

Let $\cN_3(X)$ denote the set of sampled cyclic cubic fields $K$ (having $\Cond(K) \leq X$), and let $\cN_3(X, \rho = r) \subseteq \cN_3(X)$ denote the subset of fields $K$ with $\rk_2 \Cl(K) = r$. For each of $X = 10^{5}$, $10^{6}$, and $10^{7}$, we sampled $\#\cN_3(X) = 10^4$ fields. Note that the asymptotic number of cyclic cubic fields with conductor bounded by $X$ is $c_3 \cdot X$ where $c_3 \approx 0.159$ \cite{cohn} (see also \cite[Corollary 4.7]{CDOcyclic}). We remark that in the below tables, $N =$ sample size, $\pm1/\sqrt{N}$ indicates the confidence interval, and when the prediction is a theorem, we indicate it in bold.
 
\begin{equation} \label{table:deg3data}\addtocounter{equation}{1} \notag
\begin{gathered}
\begin{tabular}{c|c|c|c|c|c}
{Family} & {Property}
 & \multicolumn{3}{c|}{{Proportion of family satisfying property}} &
{Prediction} \\
\hline\hline
 \multicolumn{2}{c|}{\phantom{$X^{X^X}$}} & \multicolumn{1}{c|}{\phantom{x.}$X = 10^{5}$\phantom{x.}} & \multicolumn{1}{c|}{\phantom{x.}$X =10^{6}$\phantom{x.}} & \multicolumn{1}{c|}{$X = 10^{7}$} & \multicolumn{1}{c}{}\\
\hline\hline
\multirow{3}{*}{\shortstack{\\ \\ \\ \\ \\ $\cN_3(X)$ \\ \\ \\ \scalebox{0.70}{$1/\sqrt{N} = .01$}}}
& $\rk_2 \Cl(K)=0$ & 0.873 & 0.871 & 0.867 &  0.853 \phantom{\rule[0pt]{0pt}{15pt}} \\
& $\rk_2 \Cl(K)=2$ & 0.127 & 0.129 & 0.133 & 0.142 \phantom{\rule[0pt]{0pt}{15pt}}\\
& $\rk_2 \Cl(K)\geq 4$ & 0.001 & 0.001 & 0.001 & 0.005 \phantom{\rule[0pt]{0pt}{15pt}}\\ 
\hline
\multirow{2}{*}{\shortstack{ \\ $\cN_3(X)$ \\ \\ \scalebox{0.70}{$1/\sqrt{N} = .01$}}}
&$\sgnrk(\calO_K^\times)=1$ & 0.023 & 0.024 & 0.026 & 0.030 \phantom{\rule[0pt]{0pt}{15pt}}  \\
& $\sgnrk(\calO_K^\times)=3$ & 0.977 & 0.976 & 0.974 & 0.970 \phantom{\rule[0pt]{0pt}{15pt}}\\ 
\hline
\multirow{2}{*}{\shortstack{\\ $\cN_3(X,\rho = 0)$ \\ \\ \scalebox{0.70}{${1}/{\sqrt{N}} \approx .11$}}}
& $\sgnrk(\calO_K^\times) =1$ & 0.000 & 0.000 & 0.000 &  {\bf 0} \phantom{\rule[0pt]{0pt}{15pt}}  \\
& $\sgnrk(\calO_K^\times) =3$ & 1.000 & 1.000 & 1.000 & {\bf 1} \phantom{\rule[0pt]{0pt}{15pt}}\\
\hline
\multirow{2}{*}{\shortstack{\\ \\ $\cN_3(X,\rho = 2)$ \\ \\ \scalebox{0.70}{${1}/{\sqrt{N}} \approx .27\hbox{-}.28$}}}
& $\sgnrk(\calO_K^\times) =1$ & 0.177 & 0.185 & 0.189 & $ 0.200 = \frac{1}{5}$\phantom{\rule[0pt]{0pt}{15pt}} \\
& $\sgnrk(\calO_K^\times) =3$ & 0.823 & 0.814 & 0.811 & $0.800 = \frac{4}{5}$\phantom{\rule[0pt]{0pt}{15pt}} \\
\end{tabular} \\
\text{Table \ref{table:deg3data}: Data for class group and signature ranks of sampled cyclic cubic fields}
\end{gathered}
\end{equation} 
\begin{equation} \label{table:deg3momentdata}\addtocounter{equation}{1} \notag
\begin{gathered}
\begin{tabular}{c|c|c|c|c|c}
{Family} & {Moment}
 & \multicolumn{3}{c|}{{Average}} &
{Prediction} \\
\hline\hline
 \multicolumn{2}{c|}{\phantom{$X^{X^X}$}} & \multicolumn{1}{c|}{$X = 10^{5}$} & \multicolumn{1}{c|}{$X =10^{6}$} & \multicolumn{1}{c|}{$X = 10^{7}$} & \multicolumn{1}{c}{}\\
\hline\hline
\multirow{3}{*}{\shortstack{\\ \\ \\ \\ \\ $\cN_3(X)$ \\ \\ \\ \scalebox{0.70}{$1/\sqrt{N} = .01$}}}
& $\#\Cl(K)[2]$ & 1.404 & 1.434 & 1.467 &  $1.500 = \frac{3}{2}$ \phantom{\rule[0pt]{0pt}{15pt}} \\
& $\left(\#\Cl(K)[2]\right)^2$ & 3.268 & 3.702 & 4.079 & $4.500 = \frac{9}{2}$ \phantom{\rule[0pt]{0pt}{15pt}}\\
& $\left(\#\Cl(K)[2]\right)^3$ & 14.76 & 21.44 & 26.67 & $40.50 = \frac{81}{2}$ \phantom{\rule[0pt]{0pt}{15pt}}\\ 
\end{tabular} \\
\text{Table \ref{table:deg3momentdata}: Data for moments of (narrow) class groups of sampled cyclic cubic fields}
\end{gathered}
\end{equation}

\subsection{Septic fields} \label{sec:septiccomp}

We now turn to computations for cyclic extensions of degree seven. The complexity of the fields grew so quickly that it was infeasible to sample fields. Instead we computed the first 8000 cyclic degree seven fields ordered by conductor. This list is available online \cite{ourlist}, and we confirmed our results against independent computations of Hofmann \cite{TommyHofmann}. 

\smallskip

Let $\cN_7(X)$ denote the set of septic cyclic fields with $\Cond(K) < X$, and let $\cN_7(X,\rho = r) \subseteq \cN_7(X)$ denote the subset of fields $K$ satisfying $\rk_2 \Cl(K) = r$.  Asymptotically, we have $\cN_7(X) \sim c_7 \cdot X$ where $c_7 \approx 0.033$ by \cite[Corollary 4.7]{CDOcyclic}. The first 8000 cyclic septic fields corresponds to the set $\cN_7(X_0)$ where $X_0 = 244861$. In addition, we have $\#\cN_7(X_0, \rho = 0) = 7739$, $\#\cN_7(X_0, \rho = 3) = 241$, and $\#\cN_7(X_0, \rho = 6) = 20$. For all other $r \in \bZ_{\geq 0}$, we have $\#\cN_7(X_0, \rho = r) = 0$. Because the sample size was so small, in Table \ref{table:deg7data} below we do not compute statistics for the subset $\cN_7(X_0,\rho = 6)$. The first few fields in $\cN_7(X_0,\rho=6)$ are generated by the roots of the polynomials:
 \begin{equation*}
     \begin{aligned}
&  x^7 - 1491x^5 + 29323x^4 - 118783x^3 - 662004x^2 + 1844864x - 899641,\\
 & x^7 + x^6 - 3360x^5 + 54087x^4 + 1523280x^3 - 24904626x^2 - 194909041x + 2439485891, \\
& x^7 - x^6 - 8274x^5 - 249021x^4 + 3000578x^3 + 60235500x^2 + 152710207x + 67428091,
 \\
 & x^7 - x^6 - 14340x^5 + 328464x^4 + 46377824x^3 - 1467892080x^2 - 11615446400x + 118681888000, \\
 & x^7 - 18543x^5 + 154525x^4 + 67057669x^3 - 1368522848x^2 - 26253624432x + 269027889901.
     \end{aligned}
     \end{equation*}

These computations took approximately 1 CPU day. Further computations quickly run into the difficulty of computing class groups of fields with large discriminants. When the prediction is a theorem, we indicate it in bold. In addition, the class group and narrow class group computations remain conjectural on GRH.  Our code is available online \cite{ourlist}.
 
\begin{equation} \label{table:deg7data}\addtocounter{equation}{1} \notag
\begin{gathered}
\begin{tabular}{c|c|c|c}
\multirow{2}{*}{Family} 
& \multirow{2}{*}{Property} & {{Proportion of family}}& \multirow{2}{*}{{Prediction}}  \\
& & {{satisfying property}} & \vspace{1pt}\\
\hline \hline 
\multicolumn{2}{c}{\phantom{$X^{X^X}$}} & \multicolumn{1}{c}{$X \approx 244861$} & \multicolumn{1}{c}{}    \\
\hline\hline
\multirow{3}{*}{\shortstack{\\ \\  \\ \\ \\ $\cN_7(X)$ \\ \\ \\ \scalebox{0.7}{$\# = 8000$}}} 
& $\rk_2 \Cl(K) = 0$ &  0.967 & ? \phantom{\rule[0pt]{0pt}{15pt}} \\
& $\rk_2 \Cl(K) = 3$ &  0.030 & ? \phantom{\rule[0pt]{0pt}{15pt}}\\
& $\rk_2 \Cl(K) \geq 6$ & 0.003 & ? \phantom{\rule[0pt]{0pt}{15pt}}\\
\hline
\multirow{3}{*}{\shortstack{ \\ $\cN_7(X,\rho = 0)$ \\ \\ \scalebox{0.7}{$\# = 7739$}}}
& $\rk_2 \Cl^+(K) = 0$ &  0.772  & $ 0.778 = \frac79$\phantom{\rule[0pt]{0pt}{15pt}} \\
& $\rk_2 \Cl^+(K) = 3$ &  0.228  & $ 0.222 = \frac29$\phantom{\rule[0pt]{0pt}{15pt}}\\
&&&\vspace{-12pt}\\
\hline
\multirow{3}{*}{\shortstack{ \\ $\cN_7(X,\rho = 0)$ \\ \\ \scalebox{0.7}{$\# = 7739$}}}
& $\sgnrk(\cO_K^\times) = 4$ &  0.228 & $ 0.222 = \frac29$\phantom{\rule[0pt]{0pt}{15pt}} \\
& $\sgnrk(\cO_K^\times) =7$ &  0.772 & $ 0.778 = \frac79$\phantom{\rule[0pt]{0pt}{15pt}} \\
&&&\vspace{-12pt}\\
\hline
\multirow{2}{*}{\shortstack{ \\ $\cN_7(X,\rho = 3)$ \\ \\ \scalebox{0.7}{$\# = 241$}}}
& $\rk_2 \Cl^+(K) = 3$ &  0.00  &  {\bf 0}\phantom{\rule[0pt]{0pt}{15pt}} \\
& $\rk_2 \Cl^+(K) = 6$ &  1.00  & {\bf 1}\phantom{\rule[0pt]{0pt}{15pt}} \\
\hline
\multirow{3}{*}{\shortstack{ \\ \\ \\ \\ \\ $\cN_7(X,\rho = 3)$ \\ \\ \\ \scalebox{0.7}{$\# = 241$}}}
& $\sgnrk(\cO_K^\times) = 1$ &  0.083 & $ 0.111 = \frac{1}{9}$\phantom{\rule[0pt]{0pt}{15pt}} \\
& $\sgnrk(\cO_K^\times) =4$ &  0.917 &  $ 0.889 = \frac{8}{9}$\phantom{\rule[0pt]{0pt}{15pt}}\\
& $\sgnrk(\cO_K^\times) =7$ &  0.000 & $ {\bf 0}$  \phantom{\rule[0pt]{0pt}{15pt}}\\
\end{tabular} \\
\text{Table \ref{table:deg7data}: Data for class group and signature ranks of the first 8000 cyclic septic fields}
\end{gathered}
\end{equation}
There are two non-trivial characters $\chi$ and $\chi^*$ for the Galois group when $G_K \simeq \Z/7\Z$. In light of Corollary \ref{cor:thmKgalcor_reflect}, one may wonder how often the inequalities 
$|\rk_\chi \Cl(K) - \rk_{\chidual} \Cl(K)| \leq 1$ and
$|\rk_\chi \Cl^+(K) - \rk_{\chidual} \Cl^+(K)| \leq 1$ 
are equalities, i.e., how often the class group or the narrow class group is {\em not} self-dual. As previously mentioned, the 2-torsion subgroup of the class group is self-dual precisely when the exponent $n$ satisfying $\rk_2 \Cl(K) = 3^n$ is \emph{even}. In particular, for the first 8000 cyclic septic fields we found that a proportion of $0.970$ have $\Cl(K)[2]$ self-dual. 

\begin{equation} \label{table:clgroup}\addtocounter{equation}{1} \notag
\begin{gathered}
\begin{tabular}{c|c|c|c} 
{Family}  & {{Moment}} & {{Average}}  & {{Prediction}} \\ 
\hline\hline
\multicolumn{2}{c}{\phantom{$X^{X^X}$}} & \multicolumn{1}{c}{$X \approx 244861$} & \multicolumn{1}{c}{} \\
\hline\hline
\multirow{3}{*}{ \shortstack{\\ \\ \\ \\ \\ $\cN_7(X)$ \\ \\ \\ \scalebox{0.7}{$1/\sqrt{N} = 0.11$}}} &
 $\#\Cl(K)[2]$ &  1.368 & $1.375 = \frac{11}{8}?$\phantom{\rule[0pt]{0pt}{15pt}} \\
 &  $\left(\#\Cl(K)[2]\right)^2$ &  13.13 & ? \phantom{\rule[0pt]{0pt}{15pt}}\\
  & $\left(\#\Cl(K)[2]\right)^3$ & 671.7 & ? \phantom{\rule[0pt]{0pt}{15pt}}\\ 
\hline
\multirow{3}{*}{ \shortstack{\\ \\ \\ \\ \\ $\cN_7(X)$ \\ \\ \\ \scalebox{0.7}{$1/\sqrt{N} = 0.11$}}} &
 $\#\Cl^+(K)[2]$ & 4.823 & ?
 \phantom{\rule[0pt]{0pt}{15pt}} \\
 &  $\left(\#\Cl^+(K)[2]\right)^2$ &  277.5 & ? \phantom{\rule[0pt]{0pt}{15pt}}\\
  & $\left(\#\Cl^+(K)[2]\right)^3$ & 75643.8 & ? \phantom{\rule[0pt]{0pt}{15pt}}\\ 
\end{tabular} \\
\text{Table \ref{table:clgroup}: Data for moments of class groups of the first 8000 cyclic septic fields}
\end{gathered}
\end{equation}

For the 2-torsion subgroup of the narrow class group, there are two possibilities: if $\Cl(K)[2]$ is not self-dual, then $\Cl^+(K)[2]$ is self-dual; and if $\Cl(K)[2]$ is self-dual we conjecture that $\Cl^+(K)[2]$ is also self-dual with probability $7/9=0.778$. Our conjecture implies that the proportion of cyclic septic fields with $\Cl^+(K)[2]$ self-dual is at least $0.778$. This data suggests that  it may be much more likely for class groups and narrow class groups to be self-dual.

\begin{remark}\label{rem:malle}
As of yet, no corrected predictions taking into account the existence of the $2$nd (but not $4$th) roots of unity in the base field have been made on the distribution of the $2$-ranks of class groups of cyclic septic fields over $\bQ$ (or more generally, of degree $\ell$ cyclic fields for any (fixed) prime $\ell \geq 7$ over $\bQ$). The original distribution in Cohen--Lenstra \cite{CohenLenstra} predicts the average size of $\Cl(K)[2]$ when $K$ varies over cyclic septic fields to be $\frac{81}{64} \approx 1.266$. However, our computations for $\ell=7$ (see Table \ref{table:clgroup}) weakly suggest that the average size of $\Cl(K)[2]$ for cyclic septic fields is $\frac{11}{8}$.

In fact, we expect that the distribution of the moments for the 2-torsion subgroups in class groups of cyclic fields of prime degree $\ell$ to be quite different when the order of $2$ modulo $\ell$ is even than when the order is odd. For example, the distribution given in Conjecture \ref{conj:quoteMallequote} implies that the average size of $\Cl(K)[2]$ is $1+2^{-f/2}$ for $\ell = 3$ or $5$ and $f$ denotes the order of $2$ modulo $\ell$. For $\ell = 7$, this computes to approximately $1.354$, which seems to diverge from Table \ref{table:clgroup}.
\end{remark}
  
\appendix

\section{Cyclic cubic fields with signature rank 1 (with Noam Elkies)} \label{sec:infincub}

In this appendix, we use Diophantine methods to construct infinite families of cyclic cubic fields with no units of mixed sign (unit signature rank $1$). 

\subsection{Setup}  

Start with a generic polynomial of the form 
\begin{equation*} \label{Equation::UnitPolynomial}
f_{a,b}(x)=f(x) \colonequals x^3 - ax^2 + bx -1, \quad \quad (a,b) \in \Z^2,
\end{equation*}
with constant coefficient $-1$; a root of $f(x)$ is a unit in $\Z[x]/(f(x))$.   
By the rational root test, the polynomial $f(x)$ is reducible over $\Q$
if and only if $b=a$ or $b=-a-2$.
When $f(x)$ is irreducible, let $K \colonequals \Q[x]/(f(x))$.  To ensure that $K$ is a cyclic cubic field, we need the discriminant $D(a,b)$ of this polynomial to be a square, i.e., we need $c \in \Z$ such that  
\begin{equation} \label{Equation::DiscriminantSurface}
c^2 = D(a,b) = -4a^3 + a^2b^2 + 18ab - 4b^3+27.
\end{equation}
The equation \eqref{Equation::DiscriminantSurface} describes a surface $S$ in $\A_{\Z}^3$ in the variables $(a,b,c)$. Several curves on the surface $S$ have been studied; for example, the \emph{simplest cubics} of Shanks \cite{Shanks} are defined by $(a,b,c)=(a,-(a+3),a^2+3a+9)$. See also work of Balady \cite{Balady} for a study of the surface $S$ over $\Q$ and other families of cubic fields arising from rational curves on $S$; unfortunately, the families he presents \cite[\S 4]{Balady} all have unit signature rank $3$. By studying the surface $S$ we prove the following theorem. 
\begin{theorem} \label{thm:cyclfields}
There are cyclic cubic fields of arbitrarily large discriminant with unit signature rank $1$. 
\end{theorem}

More precisely, we find families of Fermat--Pell curves on $S$ that have infinitely many integral points $(a,b,c) \in S(\Z)$ and such that:
\begin{equation} \label{eqn:thingwewant}
\text{the roots of $f(x)$ are totally positive and not squares of smaller units}.  
\end{equation}
Studying the ramification in these extensions proves that our procedure produces cyclic cubic fields of arbitrarily large discriminant. 

Theorem \ref{thm:cyclfields} makes a result of Dummit--Dummit--Kisilevsky \cite[Theorem 3]{DDK} unconditional: namely, the difference between $\varphi(m)/2$ and the unit signature rank of $\Q(\cos(2\pi/m))$ can be arbitrarily large.  This result has also been given a different (unconditional) proof by Dummit--Kisilevsky \cite[Theorem 7]{DK2}.

\subsection{Construction of curves}  
 
Plotting the discriminant $D(a,b)=0$ we find two curves:
\begin{equation} \label{eqn:theplot}
\includegraphics[scale=0.13,valign=c]{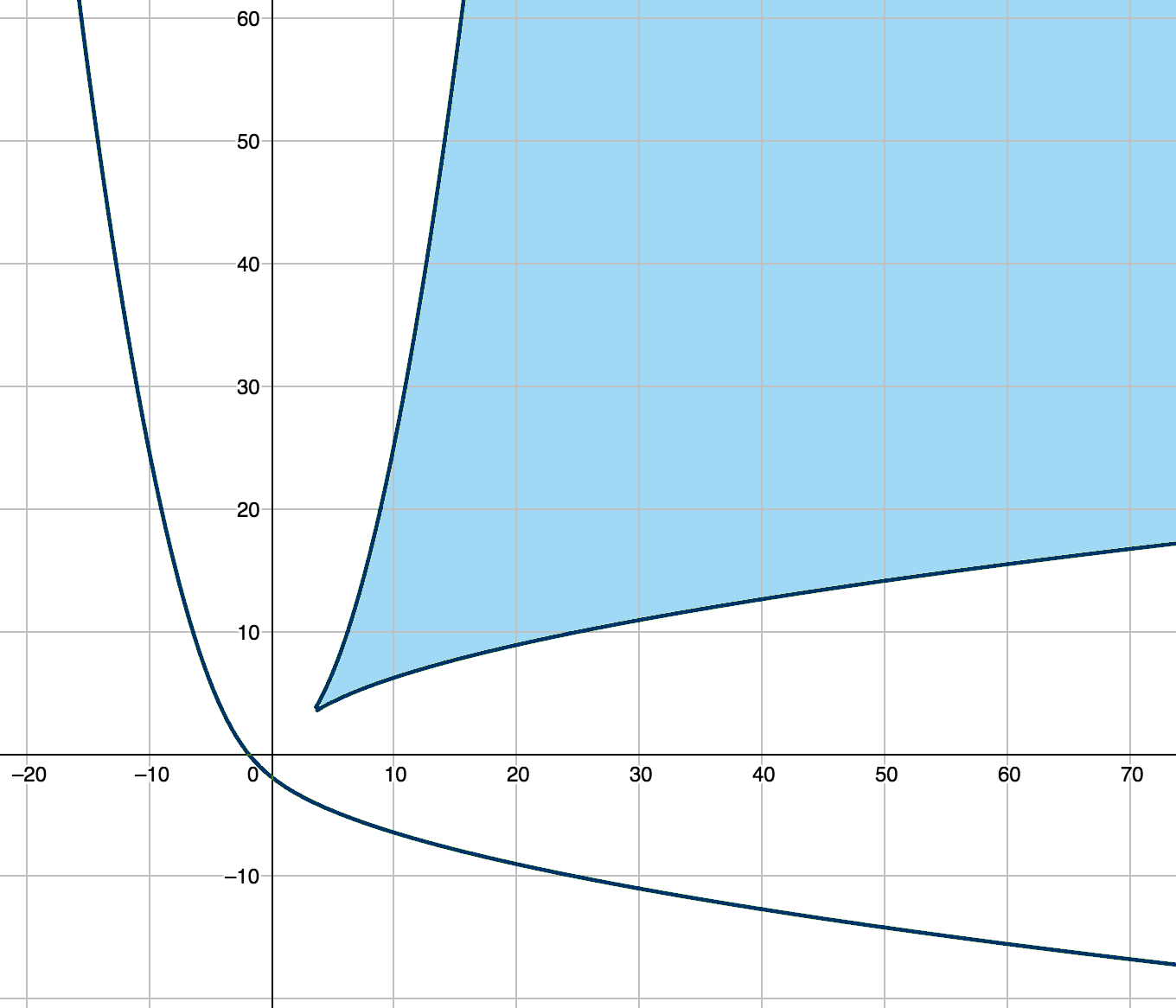}
\end{equation}
By continuity and checking values, the region in the upper right quadrant bounded by the cuspidal curve is the locus of $(a,b) \in \R^2$ with three positive roots: these are precisely the values of $(a,b)$ where $a,b>0$ and $f(x)$ has all real roots.

The curve $D(a,b)=0$ has a visible cusp at $(a,b)=(3,3)$,
corresponding to the cubic $f_{3,3}(x) = (x-1)^3$.
There are also two conjugate cusps $(a,b) = (3\zeta,3\bar\zeta)$
where $\zeta$ is one of the nontrivial cube roots of unity
$(-1 \pm \sqrt{-3}) / 2$; these likewise correspond to
$f_{3\zeta,3\bar\zeta}(x) = (x-\zeta)^3$.
The line joining these cusps is $a+b+3=0$;
on this line $D(a,b) = D(a,-(a+3))$ is a quartic in $a$ with
a double root at each $a = 3\zeta$,
and indeed $D(a,-(a+3)) = (a^2+3a+9)^2$,
so we recover Shanks' ``simplest cubics'';
we know already that we cannot use those cubics, and indeed
the line $a+b+3=0$ is disjoint from the shaded region in (\ref{eqn:theplot}).

We obtain our Fermat--Pell curves by trying curves of
the next-lowest degree passing through the conjugate cusps.
We use the pencil of parabolas in the $(a,b)$-plane
passing through those cusps and the point at infinity
$(a:b:1) = (0:1:0)$; that is, 
\begin{equation*}\label{Equation::Parabolas}
P_m \colon b = m(a^2+3a+9) - (a+3)
\end{equation*}
depending on a parameter $m \in \Q$.
On such a parabola, $D(a,b)$ is a sextic in~$a$ divisible by $(a^2+3a+9)^2$;
explicitly,
\begin{equation}
\label{Equation:Quadratic_factor}
c^2 = (a^2+3a+9)^2 Q_m(a),
\end{equation}
where $Q_m$ is the quadratic polynomial
$$Q_m(a) = m^2(1-4m)a^2 + (-12m^3 + 12m^2 - 2m) a  + (-36m^3 + 36m^2 - 12m + 1).$$
Dividing equation \eqref{Equation:Quadratic_factor} by the square factor on the right side  transforms the Diophantine equation $D(a,b)=c^2$ into an equivalent Fermat--Pell curve of the form $x^2 = Q_m(a)$ for suitable choices of~$m$. 

To avoid issues with integrality, let $m=p/q$ (written in lowest terms) and clear the denominators of \eqref{Equation:Quadratic_factor} by multiplying by $q^3$. The change of variable $y = qc/(a^2+3a+9)$ yields an equivalent integral curve
\begin{equation}
\label{Equation:PellEquation}
C_m \colon qy^2 = Aa^2 + B a + C
\end{equation}
in $\A_{\Z}^2$, where  
\begin{equation*}
\begin{aligned}
A & \colonequals  -4p^3+p^2q, \\
B & \colonequals   -12p^3 +12p^2q -2pq^2,  \\
C & \colonequals -36p^3 + 36p^2q - 12pq^2 + q^3. 
\end{aligned} 
\end{equation*}
We shall see that there are $m$ for which this curve
yields an infinite family of cyclic cubic fields with no mixed-sign units.

For starters, in order for $P_m$ to have integral points, the denominator of $m$ must be odd since $a^2 +3a + 9$ is always odd. Additionally, for $P_m$ to have infinitely many integral points in the first quadrant of the $ab$-plane, we must have $m>0$, else the intersection of $P_m$ with the half-plane $b>0$ is bounded.

\begin{proposition} \label{Prop:binteger}
Let $m \in \Q$ be such that the following conditions hold:
\begin{enumroman}
\item There exists $(a,y_0) \in C_m(\Z)$ with $a>0$ and $2Aa+B>0$; 
\item $m = p/q$ with $q$ odd;
\item $0<m<1/4$; and
\item $1-4m$ is not a square. 
\end{enumroman}
Then there exist infinitely many $(a,y) \in C_m(\Z)$ with $a>0$, and we have a map
\begin{equation*} 
\begin{aligned}
\phi_m \colon C_m(\Z) &\to S(\Z) \\
(a,y) &\mapsto (a,b,c)=(a, m(a^2+3a+9)-(a+3),y(a^2+3a+9)/q).
\end{aligned}
\end{equation*}
\end{proposition}

\begin{remark}
We note that condition (i) implies condition (ii) in Proposition \ref{Prop:binteger}. This can be shown by checking 2-adic valuations of each side of Equation \eqref{Equation:PellEquation}.
\end{remark}

\begin{proof}
Completing the square in \eqref{Equation:PellEquation}, we obtain the standard form:
\begin{equation} \label{eqn:xay2}
x^2-(4Aq)y^2 = B^2-4AC
\end{equation}
where $x=2Aa+B$.  By (iii), we have $0<m^2(1-4m)=A/q^3$ so $4Aq>0$ and $A>0$.  By (i), there exists $(a,y_0) \in C_m(\Z)$ with $a>0$, which gives a point $(x_0,y_0) \in \Z^2$ on \eqref{eqn:xay2} with $x_0=2Aa+B>0$.  Changing signs, we may suppose without loss of generality that $y_0>0$.  By (iv), we conclude $4Aq=(1-4m)(2mq^2)^2$ is not a square.  

We now apply the theory of Pell equations to obtain infinitely many solutions $(x,y) \in \Z^2$ to \eqref{eqn:xay2} with $x \equiv x_0 \equiv B \pmod{2A}$ and $x>0$: explicitly, there exists a power of the fundamental unit for the real quadratic field $\Q(\sqrt{Aq})$ of the form $\epsilon = r+s\sqrt{4Aq}$ with $r,s \in \Z_{>0}$ and $r \equiv 1 \pmod{2A}$, so the solutions $(x_k,y_k)$ obtained by multiplying $x_0+y_0\sqrt{4Aq}$ by the powers $\epsilon^k = r_k + s_k\sqrt{4Aq}$ for $k \geq 1$ have $r_k \equiv 1 \pmod{2A}$ and $r_k,s_k>0$, so $x_k=r_kx_0+s_ky_0(4Aq) \equiv x_0 \pmod{2A}$ and $x_k>x_0>0$.  Thus 
\[ a_k=(x_k-B)/(2A)>\cdots>(x_0-B)/(2A)=a>0, \]
and we obtain infinitely many points $(a_k,y_k) \in C_m(\Z)$.  

To conclude, we claim that if $(a,y) \in C_m(\Z)$, then  
$$ b = m(a^2+3a+9) - (a+3) \in \Z.$$
Reducing \eqref{Equation:PellEquation} modulo $q$ gives 
\begin{equation*}
    0 \equiv -4p^3a^2-12p^3a-36p^3 = -4p^3(a^2+3a+9) \pmod{q};
\end{equation*}
and since $q$ is odd and $\gcd(p,q)=1$ we conclude $q \mid (a^2+3a+9)$, so $m(a^2+3a+9) \in \Z$ and consequently $b \in \Z$.
\end{proof}

\begin{remark}
The method for getting infinitely many $(a,y) \in C_m(\Z)$ from
an initial solution was already known to Euler \cite{Euler};
see Dickson \cite[pp.~355--356]{Dickson} 
  (English translation in the Euler Archive, \url{http://eulerarchive.maa.org/tour/tour_12.html}).
  Dickson describes Euler's technique, which comes down to the same
  construction, though of course Euler did not use the arithmetic of
  real quadratic number fields.
\end{remark}

To find a value of $m$ suitable for applying Proposition \ref{Prop:binteger}, we work backwards by first selecting an integral point $(a,b,c) \in S$ (by a brute force search or starting with a cyclic cubic field of unit signature rank $1$) and then solving for the parameter $m$ of the parabola $P_m$.  (Since $m$ occurs linearly in the formula for $P_m$, there is a unique solution; explicitly
\begin{equation} \label{eqn:m}
m = \frac{a+b+3}{a^2+3a+9}.
\end{equation}
As it happens the denominator is always positive so we do not even
have to worry about dividing by zero at an unfortunate choice of $(a,b)$.)

\begin{example} \label{exm:ab1}
Let $(a,b) = (149,4018)$. Solving for the corresponding parabola yields $m = 30/163$ which satisfies the conditions on $m$. The resulting equation is
\begin{equation*}
    C_m \colon 163 y^2 = 38700a^2  -157740a - 924893
\end{equation*}
and yields a sequence of solutions 
\begin{align*}
    (a,b) &= (149, 4018), (395449, 28781401718), \\
 & \qquad (655993191035058918, 79201300616753245838398841511537549), \dots 
\end{align*}
\end{example}

\begin{example}  \label{exm:ab2}
Similarly for $(a,b) = (269, 10986)$, we obtain $m = 2/13$, 
$$C_m \colon 13 y^2 = 20a^2 - 148a - 275,$$
and 
\begin{align*}
    (a,b) &= (1725, 456858), (17657181,47965535241018), \\
 & \qquad (114572909, 2019530934725706), \\
 & \qquad (1175297035181, 212511249369405417243018), \dots
 \end{align*}
\end{example}

\begin{remark}
Having found one $m$ satisfying the conditions of
Proposition~\ref{Prop:binteger}, such as $m = 2/13$ above,
we can find infinitely many more.
This is because $D$, and thus $S$,
is symmetric under $(a,b) \leftrightarrow (b,a)$.
Given an infinite sequence of $(a_k,b_k) \in C_m(\Z)$,
we may switch $a,b$ in (\ref{eqn:m}) to find an infinite sequence of
$m_k = (a_k + b_k + 3) / (b_k^2 + 3 b_k + 9)$ satisfying
condition (i) of Proposition~\ref{Prop:binteger}.
For $m=2/13$, these $m_k$ begin
$$
\frac2{21447},\;
\frac2{910279},\;
\frac2{95931035167687},\;
\frac2{4039061640305607},\;
\frac2{425022498736460240415687},\dots
$$
corresponding to the initial solution $(a,b) = (149,4018)$ in
Example~\ref{exm:ab1} and the further four solutions listed there.
Condition (ii), that $0 < m_k < 1/4$, is satisfied for
all but finitely many $k$: each $m_k$ is positive, and
$m_k \to 0$ because $a_k \to \infty$ and $b_k \sim m a_k^2$.
It remains to check that $1-4m_k$ is not a square for infinitely many~$k$
(condition~(iii)).  This can be done in various ways;
for example, once we have checked this for one $k_0$,
we can find some prime $\ell$ such that $1-4m_{k_0}$
is not a square mod~$\ell$, and apply Euler's theorem
as in~\ref{Prop:binteger} to find infinitely many $k$
such that $m_k \equiv m_{k_0} \pmod \ell$,
whence $1 - 4m_k$ is not a square either.
For our $m=2/13$ we may use $m_{k_0} = 2/21447$ and $\ell=5$.
This gives infinitely many curves $C_{m_k}$ each containing
infinitely many integral points of~$S$\/ above the shaded region
in~(\ref{eqn:theplot}), thus showing that such points are
Zariski-dense in~$S$.
(This is the same technique used by Elkies \cite{Elkies:Euler}
to find a Zariski-dense set of rational points on the Fermat
quartic surface $A^4 + B^4 + C^4 = D^4$ starting from a single
elliptic curve on that surface with infinitely many rational points.)
\end{remark}

\subsection{Infinitely many cyclic cubic fields}  

The construction above produces infinitely many integral points $(a,b)$ that correspond to cyclic cubic fields with totally positive units.  We now show that for all but finitely many $(a,b)$, the condition \eqref{eqn:thingwewant} holds.  

For $a,b \in \Z^2$ such that $f_{a,b}(x)=x^3-ax^2+bx-1$ is irreducible, let $K_{a,b} \colonequals \Q[x]/(f_{a,b}(x))$ and let $\eta_{a,b} \in K_{a,b}$ be the image of $x$.

\begin{lemma}\label{lem:squares}
The following statements hold.  
\begin{enumalph}
\item If $\eta_{a,b} \in K_{a,b}^{\times 2}$, then $(a,b) = (A^2-2B, B^2-2A)$ for some $A,B \in \Z$.
\item Let $m \in \Q$.  Then there are only finitely many $(a,b) \in P_m(\Z)$ such that $\eta_{a,b} \in K_{a,b}^{\times 2}$.  
\end{enumalph} 
\end{lemma}

\begin{proof}
For (a), let $\epsilon^2 = \eta_{a,b}$. Replacing $\epsilon$ by $-\epsilon$ if necessary, we may suppose that $\epsilon$ is a root of $x^3-Ax^2+Bx-1$.  Expressing $a,b$ as symmetric polynomials in the roots, we obtain the result.

For (b), we study the squares on the parabola $P_m$ by substituting in (a) to get
$$ B^2-2A = m ((A^2-2B)^2 + 3(A^2-2B) + 9) -(A^2-2B +3). $$
which yields 
\begin{equation} \label{eqn:smallm}
(1-4m)B^2 + (4mA^2+6m-2)B = mA^4 + (3m-1)A^2 + 2A + (9m-3);
\end{equation}
multiplying by $1-4m$ and letting $B' \colonequals (1-4m)B$, we obtain
\begin{equation} \label{eqn:hyperell}
B'^2 + (4mA^2+6m-2)B' = (1-4m)(mA^4 + (3m-1)A^2 + 2A + (9m-3)).
\end{equation}
The equation \eqref{eqn:hyperell} describes a family of curves of genus $1$ in the variables $A,B$ over $\Q(m)$.  Its discriminant is $-256m(1-4m)^2(27m^2-9m+1)^3$, so \eqref{eqn:hyperell} is smooth of genus $1$ for all $m \in \Q \smallsetminus \{0,1/4\}$, and so the same is true of \eqref{eqn:smallm}.  By Siegel's theorem \cite[Corollary IX.3.2.2]{Silverman}, for fixed $m$ the equation \eqref{eqn:smallm} has finitely many integral solutions $(A,B) \in \Z^2$.  The conclusion then follows from (a).
\end{proof}

Next, we prove that the construction above produces infinitely many distinct fields of the form $K_{a,b}$.  Let $\alpha_{a,b} \colonequals 3\eta_{a,b} -a \in \mathcal{O}_{K_{a,b}}$; then $\alpha_{a,b}$  is an algebraic integer satisfying:
\begin{equation*} \label{eqn:shifted} 
M_{a,b} (x) \colonequals x^3 + (a^2+3a+9)(9m-3)x + (a^2+3a+9)(a(9m-2)-3) \in \Z[x].
\end{equation*}

\begin{lemma} \label{lem:Inftyfields}
Let $m \in \Q$ satisfy conditions \textup{(i)}--\textup{(iv)} of Proposition \textup{\ref{Prop:binteger}}.  Then as $a$ ranges over points $(a,b,c) \in \phi_m(C_m(\Z)) \subseteq S(\Z)$ with $a>0$, the set of primes $\ell$ such that $3 \nmid \ord_\ell(a^2+3a+9)$ is infinite.  Moreover, $\alpha_{a,b}$ generates a totally ramified extension of $\Q_\ell$ for all but finitely many such $\ell$ (depending on $m$).  
\end{lemma}

\begin{proof} 
Let $t \in \Z_{>0}$ be cubefree.  Then $a^2+3a+9=tz^3$ defines a genus $1$ curve, so by Siegel's theorem it has finitely many integral points.  Therefore  there are only finitely many $(a,y) \in C_m(\Z)$ such that $a^2+3a+9=tz^3$ for $z \in \Z$. But $\#C_m(\Z)=\infty$ by Proposition \ref{Prop:binteger}, so the cubefree part of $a^2+3a+9$ must take on infinitely many values. 

For the second statement, we consider the Newton polygon of $M_{a,b}(x)$ at $\ell$.  Since
\[ (81m^2-36m+4)(a^2+3a+9) + ((2-9m)a+3-27m)((9m-2)a-3)=27(27m^2-9m+1) \]
it follows that, for any prime $\ell$ such that $\ell \nmid 27(27m^2-9m+1)q^2$, we have
$$\ord_{\ell}[(a^2+3a+9)(a(9m-2)-3)] = \ord_{\ell}(a^2+3a+9).$$ 
For such primes, the $\ell$-Newton polygon of $M_{a,b}(x)$ consists of a single segment of slope $\ord_\ell(a^2+3a+9)/3$, and hence the extension defined by $M_{a,b}(x)$ over $\Q_\ell$ is totally ramified.
\end{proof}

We finish with a proof of the theorem in this section.

\begin{proof}[Proof of Theorem \textup{\ref{thm:cyclfields}}]
Let $m \in \Q$ satisfy \textup{(i)}--\textup{(iv)} of Proposition \textup{\ref{Prop:binteger}}, so that $\phi_m(C_m)(\Z)$ contains infinitely many points $(a,b,c) \in S(\Z)$ with $a > 0$, and hence $b>0$; for example, we may take $m=30/163,2/13$ as in Examples \ref{exm:ab1} and \ref{exm:ab2}.  The intersection of $P_m$ with the lines $b=a$ and $b =a-2$ removes at most $4$ values of $a$; for the values that remain, $f_{a,b}(x)=x^3-ax^2+bx-1$ is irreducible over $\Q$.  To each of these points we associate the field $K_{a,b}=\Q(\eta_{a,b})$ where $\eta_{a,b}$ is a root of $f(x)$, and consider the set of fields 
\[\mathcal{K}_m \colonequals \{ K_{a,b} : (a,b,c) \in \phi_m(C_m)(\Z) \text{ and $f_{a,b}(x)$ is irreducible} \}. \]
Each $K_{a,b} \in \mathcal{K}_m$ is a cyclic cubic extension because its discriminant is (up to squares) equal to $c^2$, and since $a,b>0$ its roots are totally positive as in \eqref{eqn:theplot}.
By Lemma \ref{lem:Inftyfields}, there are infinitely many primes $\ell$ dividing the discriminants of the fields in $\mathcal{K}$ and so the set contains fields with arbitrarily large discriminants.  By Lemma \ref{lem:squares}, in the set $\mathcal{K}$ there are only finitely many fields where $\eta_{a,b} \in K_{a,b}^{\times 2}$; let $\mathcal{K}^+$ be the infinitely many remaining fields. Since $\eta_{a,b} \not \in K_{a,b}^{\times2}$, then $\eta_{a,b}$ is a totally positive unit that is not a square.  By Corollary \ref{cor:Kcycsgrkyup}, we have $\sgnrk \calO_{K_{a,b}}^\times=1,3$, so we must have unit signature rank $1$, i.e., there is a basis of totally positive units.  
\end{proof}

\end{document}